\documentclass[final]{siamltex}
\usepackage{amssymb}
\usepackage{graphics}
\usepackage{caption2}
\usepackage{psfrag}
\usepackage{amsmath}
\usepackage{amscd}
\usepackage{amsfonts}
\usepackage{bm}                                                %
\usepackage{float}
\usepackage{latexsym}
\usepackage{lscape}
\usepackage{mathrsfs}
\usepackage{verbatim}
\newtheorem{remark}{Remark}[section]

\newcommand{\ds}{\displaystyle}
\newcommand{\f}{\frac}

\newcommand{\p}{\partial}
\newcommand{\fff}{\mbox{\boldmath $f$}}
\newcommand{\ggb}{\mbox{\boldmath $g$}}
\newcommand{\uu}{\mbox{\boldmath $u$}}
\newcommand{\xx}{\mbox{\boldmath $x$}}
\newcommand{\yy}{\mbox{\boldmath $y$}}
\newcommand{\xib}{\mbox{\boldmath $\xi$}}
\begin{document}

\title{Wellposedness and regularity of steady-state two-sided variable-coefficient conservative space-fractional diffusion equations
\thanks{This work was partially supported by the National Natural Science Foundation of China under Grants 91130010, 11471194, and 11571115, and the Taishan research project of Shandong Province, and the National Science Foundation under Grants EAR-0934747 and DMS-1216923, and the OSD/ARO MURI Grant W911NF-15-1-0562.}}
\author{Danping Yang \thanks{Department of Mathematics, East China Normal University, Shanghai,
200241, China. E-mail address: dpyang@math.ecnu.edu.cn.}
\and Hong Wang \thanks{Department of Mathematics, University of South Carolina, Columbia,
South Carolina 29208, USA. E-mail address: hwang@math.sc.edu.}}

\maketitle
\begin{abstract}
We study the Dirichlet boundary-value problem of steady-state two-sided variable-coefficient conservative space-fractional diffusion equations. We show that the Galerkin weak formulation, which was proved to be coercive and continuous for a constant-coefficient analogue of the problem, loses its coercivity. We characterize the solution to the variable-coefficient problem in terms of the solutions of second-order diffusion equations along with a two-sided fractional integral equation. We then derive a Petrov-Galerkin formulation for this problem and prove that the weak formulation is weakly coercive and so the problem is well posed. We then prove high-order regularity estimates of the true solution in a properly chosen norm of Riemann-Liouville  derivatives.
\end{abstract}



\pagestyle{myheadings}
\thispagestyle{plain}

\markboth{Yang and Wang}{Analysis on variable-coefficient fractional diffusion equations}

\pagestyle{myheadings}
\thispagestyle{plain}

\markboth{Wang and Yang} {Analysis on fractional elliptic equation}

\section{Introduction}

In recent years nonlocal models are emerging as powerful tools for modeling challenging phenomena including overlapping microscopic and macroscopic scales, anomalous transport, and long-range time memory or spatial interactions in nature, science, social science, and engineering \cite{MetBar,MetKla00,MetKla04,Wes}. 
Data-driven fractional-order differential operators can be constructed to model a specific phenomenon instead of the current practice of tweaking the coefficients that multiply pre-set integer-order differential operators. It was shown that the misspecification of physical models using an integer-order partial differential equation often leads to a variable coeﬃcient fit (struggling to fit the data at each location, for example) whereas a physical model using a fractional-order partial differential equation can fit all the data with a constant coefficient \cite{BenSch}. In short, nonlocal models open up great opportunities and flexibility for modeling and simulation of multiphysical phenomena, e.g. from local to nonlocal dynamics \cite{Wes}. Because of their significantly improved modeling capabilities, various related but different nonlocal models, including fractional Laplacian, nonlocal diffusion and peridynamics, and fractional partial differential equations, have been developed to describe diverse nonlocal phenomena. 

The fractional Laplacian operator $(-\Delta u)^s$ of order $0 < s < 1$ has been used to model nonlocal behavior in many physical problems \cite{App,BaoJia,DuoZha,Las} and has appeared as the infinitesimal generator of a stable L\'{e}vy process \cite{App,GaoDua1,GaoDua2,Val}. $(-\Delta)^s$ can be defined as a pseudodifferential operator of symbol $|\xib|^{2s}$ on the entire space $\mathbb{R}^d$ \cite{App}
\begin{equation}\label{FracLap:e1}
(-\Delta)^{s} u = \mathcal{F}^{-1} \big ( | \xib |^{2s} \mathcal{F}{u}(\xib) \big), \qquad \forall u \in \mathcal{S}
\end{equation}
where $\mathcal{S}$ refers to the Schwartz space and $\mathcal{F}$ denotes the Fourier transform \cite{AdaFou}. It can equivalently be defined by the prescription \cite{NezPal}
\begin{equation}\label{FracLap:e2}
(-\Delta)^{s} u(\xx) = C(d,s) ~\mathrm{P.V.} \int_{\mathbb{R}^d} \f{u(\xx)-u(\yy)}{|\xx-\yy|^{d+2s}}d\yy, 
\end{equation}
where the parameter $C(d,s)$ depends on the space dimension $d$ and the order $s$ of the fractional Laplacian. The $(-\Delta)^{s}$ can be extended to an integer-order partial differential equation on the half space
$\mathbb{R}_+^{d+1}$ via a Dirichlet-to-Neumann mapping \cite{CafSil}.

However, subtlety occurs in the corresponding ``boundary value" problem of the fractional Laplacian when the domain $\Omega$ under consideration is bounded, as there are more than one defitions of $(-\Delta)^{s}$ in the literature which are not necessarily equivalent. A feasible definition is to restrict the function $u$ in (\ref{FracLap:e2}) to those supported in $\Omega$. And the corresponding boundary value problem is formulated as 
\begin{equation}\label{FracLap:e3}\begin{array}{rcll}
(-\Delta)^s u(\xx) &=& f(\xx), \quad & \xx \in \Omega, \\[0.025in]
u(\xx) &=& 0, & \xx \in \Omega^c = \mathbb{R}^d\backslash \Omega.
\end{array}\end{equation}
By the Feynman-Kac formula \cite{App,Oks}: $u(x)$ can be obtained by an ensemble of the boundary data at the feet of the sample paths of a stochastic L\'{e}vy process that start from $\xx$ and just jump out of domain $\Omega$. As the sample paths of a L\'{e}vy process admit jumps of arbitrary lengths, the boundary data must be imposed on the entire complement $\Omega^c$ of the domain $\Omega$. 
On the other hand, for the Laplacian equation ((\ref{FracLap:e3}) with $s=1$), the underlying stochastic process is a Brownian motion that has continuous sample paths that intersect the boundary $\p \Omega$ of the domain $\Omega$ almost surely. Hence, the boundary condition needs only be specified at the boundary $\p \Omega$.

Alternatively, let $\{\lambda_n,\psi_n\}_{n=1}^\infty$ be the set of eigenvalues and ($L^2$ orthogonal and) normalized eigenfunctions of the Laplace operator in $\Omega$ with the homogeneous Dirichlet boundary condition on $\p \Omega$.
\begin{equation}\label{FracLap:e4}
(-\Delta)^s u = \sum_{n=1}^\infty (u,\psi_n) \lambda_n^s \psi_n, \quad \forall u \in C^\infty_0(\Omega).
\end{equation}
In \cite{StiTor} $(-\Delta)^s$ defined in (\ref{FracLap:e4}) was extended to a integer-order partial differential equation posed on $\Omega \times (0,\infty)$ by generalizing the result in \cite{CafSil}. This result was then utilized in \cite{NocOta} in the numerical approximation of the fractional Laplacian $(-\Delta)^s$ defined in (\ref{FracLap:e4}), by solving the integer-order equation on $\Omega \times (0,\infty)$ via graded meshes in the extended variable. An alternative numerical discretization of the fractional Laplacian defined by (\ref{FracLap:e4}) was presented in \cite{YanLiu} via a discrete version of the spectral decomposition of (\ref{FracLap:e4}).

The constitutive models in peridynamics depend on finite deformation vectors, instead of deformation gradients in classical constitutive models \cite{Sil00,SilEpt}. Consequently, peridynamic models yield nonlocal mathematical formulations that are based on long-range interactions and present more appropriate representation of discontinuities in displacement fields and the description of cracks and their evolution in materials than classical continuum solid mechanics that are based on local interactions.

For instance, a bond-based linear peridynamic model takes the form \cite{CheGun,Sil00,SilEpt}
\begin{equation}\label{Peri:e1}\begin{array}{rlll}
\ds C(d,\delta,k) \int_{B_{\delta}(\xx)} \frac{(\xx - \yy) \otimes (\xx - \yy)}{|\xx - \yy|^3} \big [ \uu(\xx) - \uu(\yy) \big] d \yy
& = \fff(\xx), &\xx\in \Omega,\\
\uu(\xx) &= \ggb(\xx), ~& \xx \in \Omega_\delta.
\end{array}\end{equation}
Here $\uu$ is the displacement vector, $\fff$ is the prescribed body force density field, $\Omega_\delta$ denotes a boundary zone surrounding $\Omega$ with width $\delta > 0$, and $\ggb$ is the prescribed displacement imposed on the domain $\Omega_\delta$. The constant $C$ depends on the space dimension $d$, the radius $\delta$, and the bulk modulus $k$.  The material horizon $B_{\delta}(\xx)$ is a closed ball centered at $\xx$ with the radius $\delta$. 

In other words, all the interactions in peridynamic models are allowed to be nonlocal, they are indeed assumed to be short ranged so the particle at $\xx$ does not have any interaction with particles outside of $B_\delta(\xx)$. Moreover, the ``boundary condition" is imposed neither on the classical boundary $\p \Omega$ of the domain $\Omega$ nor the entire complement $\Omega^c$ of $\Omega$ as in 
(\ref{FracLap:e3}), but rather on the ``collar" $\Omega_\delta$ of the domain $\Omega$.

In the context of one space dimension $d=1$, the peridynamic model (\ref{Peri:e1}) reduces to the nonlocal diffusion model \cite{DelGun,DuGun} which corresponds to (\ref{FracLap:e3}) with $d=1$ and $s=0$ and $\mathbb{R}^d$ being replaced by $B_\delta(x)$.

A variable-coefficient peridynamic model was derived in \cite{MenDu} in which an extra coefficient $K(\xx) + K(\yy)$ appears in the integrand in (\ref{FracLap:e2}). A  variable-coefficient analogue $(-\nabla (K(\xx) \nabla ))^s$ of (\ref{FracLap:e4}) is defined by the right-hand side of (\ref{FracLap:e4}) except that the  $\{\lambda_n,\psi_n\}_{n=1}^\infty$ are now the set of eigenvalues and eigenfunctions of the operator $-\nabla (K(\xx) \nabla)$ in $\Omega$ with the homogeneous Dirichlet boundary condition on $\p \Omega$. 

Finally, we turn to fractional differential equations (FDEs). Classical Fickian diffusion equation was derived under the assumptions of (i) the existence of a mean free path and (ii) the existence of a mean waiting time in the underlying identical and independently distributed random particle jumps \cite{Ein,Fic,Pea}. Under these assumptions, long walks in the same direction are rare so the variance of a particle excursion distance is finite. The central limit theorem concludes that the probability density function of finding a particle somewhere in space satisfies a canonical Fickian diffusion equation and thus gives rise to a probabilistic description of a normal diffusion precess \cite{MeeSik}.

However, the random particle movements in heterogeneous media often undergo long jumps and so violate the assumptions (i) and (ii). These processes may have arbitrarily long jumps and so have large deviations from the stochastic process of Brownian motion.  This is the reason why these processes cannot be described appropriately by the second-order diffusion equation. Consequently, the probability density function of finding a particle somewhere in space satisfies a L\'{e}vy distribution, which satisfies a space-fractional diffusion equation and thus gives rises a probabilistic description of an anomalous diffusion process. This explains why FDEs provide a more appropriate description of anomalous diffusion processes than classical Fickian diffusion equation and why FDE models have been used in many applications \cite{MeeSik,MetKla00,MetKla04,Pod,Wes}. 

We take the two-sided variable-coefficient conservative Caputo space-fractional diffusion equation as an example to demonstrate the idea
\begin{equation}\label{Model:e1}\begin{array}{c}
\ds - D \bigl [ K(x) \bigl (\theta ~{}^{l}_aI_x^{\beta} D + (1-\theta) ~{}^{r}_xI_b^{\beta}D \bigr )u \big] = f(x), \quad  x \in (a,b), \\[0.05in]
u(a) = u(b) = 0.
\end{array}\end{equation}
Here $D$ is the first-order differential operator, $2-\beta$ with $0 < \beta < 1$ represents the order of anomalous diffusion of the problem, $0 < K_m \le K(x) \le K_M <+\infty$ is the diffusivity coefficient, and $0 \le \theta \le 1$ indicates the relative weight of forward versus backward transition probability, and $f$ is the source term \cite{BenWhe,CasCar,MeeSik,WheMee,ZhaBen}. The left and right fractional integrals of order $\sigma > 0$ are defined for any $w \in L^1(a,b)$ by \cite{AdaFou,Pod,SamKil}
\begin{equation}\label{Model:e2}
{}^{l}_aI_x^{\sigma}w(x) := \f1{\Gamma(\sigma)} \int_a^x \f{w(s)}{(x-s)^{1-\sigma}}ds, \quad {}^{r}_xI^{\sigma}_bw(x) := \f1{\Gamma(\sigma)} \int_x^b \f{w(s)}{(s-x)^{1-\sigma}}ds
\end{equation}
where $\Gamma(\cdot)$ is the Gamma function. (\ref{Model:e1}) is derived by combining a conventional mass balance law in terms of the flux $J$
\begin{equation}\label{Model:e3}
D J = f
\end{equation}
with the fractional Fick's law that accounts for the contributions of the particles that jumps to $x$ from any point in the domain $(a,b)$ \cite{SchBen}
\begin{equation}\label{Model:e4}
J = - K(x) \bigl (\theta ~{}^{l}_aI_x^{\beta} D + (1-\theta) ~{}^{r}_xI_b^{\beta}D \bigr )u. 
\end{equation}

Although FDEs share some common mathematical properties with fractional Laplacian, peridynamics and nonlocal diffusion models due to their common nonlocality, the wellposedness of the boundary-value problems of the FDEs is more subtle to analyze than those of the other nonlocal models partly due to the following reason: The Fractional Laplacian operators $(-\Delta)^s$ defined in (\ref{FracLap:e2}) or (\ref{FracLap:e4}) and the peridynamic (and nonlocal diffusion) models (\ref{Peri:e1}) as well as their variable-coefficient analogues can be formulated as a minimization of an quadratic energy functional and are symmetric and coercive with respect to appropriate (possibly weighted) fractional Sobolev spaces \cite{AdaFou,DelGun,DuGun,Tho}. Hence, the Lax-Milgram theorem concludes that the corresponding boundary-value problems are wellposed exactly like in the context of integer-order PDEs \cite{Eva}. 

However, in the FDE in (\ref{Model:e1}) the external operator is a first-order differential operator $D$ and the internal operator is a two-sided fractional differential operator of order $1-\beta$. Hence, the FDE problem (\ref{Model:e1}) cannot be formulated as the minimization of an energy functional and the inherent trial and test spaces are fundamentally different especially in the presence of a variable diffusivity coefficient $K$. This makes the analysis of the FDE problem (\ref{Model:e1}) difficult to analyze. 

In a pioneer and foundational work on the wellposedness of problem (\ref{Model:e1}) with a constant diffusivity coefficient $K$, Ervin and Roop \cite{ErvRoo05} derived a Galerkin weak formulation and proved that its bilinear form is coercive and bounded on the product space $H^{1-\beta/2}_0(a,b) \times H^{1-\beta/2}_0(a,b)$ even though the problem cannot be formulated as the minimization of an energy functional. 
Thus, the Lax-Milgram theorem concludes that the problem is well posed \cite{ErvRoo05,Eva}. 

In this paper we show that the bilinear form of the Galerkin formulation may lose its coercivity for problem (\ref{Model:e1}) with a variable diffusivity coefficient. Numerical results show that the Galerkin finite element method does not necessarily converge in this case \cite{WanYanZhu}! We then characterize the solution to problem (\ref{Model:e1}) in terms of the solutions of second-order diffusion equations along with an integral equation. We then accordingly derive a Petrov-Galerkin formulation for problem (\ref{Model:e1}) and prove that the bilinear form of the Petrov-Galerkin weak formulation is weakly coercive and so problem (\ref{Model:e1}) is well posed. 

\section{Preliminaries} 

Let $C[a,b]$ denote the space of continuous functions on $[a,b]$, and $C^\infty_0(a,b)$ denote the space of infinitely many times differentiable functions on $(a,b)$ that are compactly supported in $(a,b)$. Let $L^p(a,b)$, with $1 \le p \le +\infty$, be the Banach spaces of $p$-th power Lebesgue integrable functions on $(a,b)$ and $W^{m,p}(a,b)$ be the Sobolev spaces of functions whose weak derivatives up to order $m$ are in $L^p(a,b)$ \cite{AdaFou,SamKil}. For any $\mu > 0$, define the (semi) norms \cite{AdaFou,ErvRoo05}
$$ | w |_{H^{\mu}(\mathbb{R})} := \big \| |\omega|^\mu \mathcal{F}(w) \big \|_{L^2(\mathbb{R})},
\qquad \| w \|_{H^{\mu}(\mathbb{R})} :=  \big (\|w\|^2_{L^2(\mathbb{R})} + |w|_{H^{\mu}(\mathbb{R})}^2 \bigr)^{1/2},$$
and the fractional Sobolev space $H^{\mu}(\mathbb{R})$ is the completion of $C^\infty_0(\mathbb{R})$ with respect to the norm $\| \cdot \|_{H^{\mu}(\mathbb{R})}$.

Let $0 < \sigma < 1$, $m$ be a postive integer, and $\mu := m - \sigma$. Then the left and right (Caputo) fractional derivatives of order $\mu$ are defined to be
$${}_a^l D_x^{\mu}w := {}_a^lI_x^{\sigma} D^m w, \quad {}_x^r D_b^{\mu}w := {}_x^r I_b^{\sigma}~ (-D)^m w.$$
We then introduce the corresponding (semi) norms for the left, right, and weighted two-sided fractional derivatives, respectively for $w \in C^\infty_0(\mathbb{R})$ \cite{ErvRoo05,SamKil}:
\begin{equation}\label{FDE:e2}\begin{array}{ll}
|w|_{J_l^{\mu}(\mathbb{R})} &:= \big \| {}_{-\infty}^{l}D_x^{\mu}w \big \|_{L^2(\mathbb{R})}, \qquad \|w\|_{J_l^{\mu}(\mathbb{R})} := \big (\|w\|^2_{L^2(\mathbb{R})} +  |w|_{J_l^{\mu}(\mathbb{R})}^2 \bigr)^{1/2}, \\[0.075in]
|w|_{J_r^{\mu}(\mathbb{R})} &:= \big \| {}_x^r D_\infty^{\mu}w \big \|_{L^2(\mathbb{R})}, ~~ \qquad \|w\|_{J_r^{\mu}(\mathbb{R})} := \big (\|w\|^2_{L^2(\mathbb{R})} +  |w|_{J_r^{\mu}(\mathbb{R})}^2 \bigr)^{1/2}, \\[0.075in]
|w|_{J^{\mu,\theta}(\mathbb{R})} &:= \big (\theta^2 | w |^2_{J_l^{\mu}(\mathbb{R})}+(1-\theta)^2 | w |^2_{J_r^{\mu}(\mathbb{R})} \bigr)^{1/2},\\[0.075in]
\|w\|_{J^{\mu,\theta}(\mathbb{R})} &:=  \big (\|w\|^2_{L^2(\mathbb{R})} + |w|_{J^{\mu,\theta}(\mathbb{R})}^2 \bigr)^{1/2}.
\end{array}\end{equation}
Let $J_l^{\mu}(\mathbb{R})$, $J_r^{\mu}(\mathbb{R})$, and $J^{\mu,\theta}(\mathbb{R})$ denote the spaces that are completion of $C^\infty_0(\mathbb{R})$ with respect to the norms $\| \cdot \|_{J_l^{\mu}(\mathbb{R})}$, $\| \cdot \|_{J_r^{\mu}(\mathbb{R})}$, and $\| \cdot \|_{J^{\mu,\theta}(\mathbb{R})}$, respectively. In addition, we let $J^{\mu}_{l,0}(a,b)$,  $J^{\mu}_{r,0}(a,b)$, $J^{\mu,\theta}_0(a,b)$, and $H^\mu_0(a,b)$ denote the function spaces that are the completion of $C^\infty_0(a,b)$ with respect to the norms $\| \cdot \|_{J^{\mu}_l(\mathbb{R})}$, $\| \cdot \|_{J^{\mu}_r(\mathbb{R})}$, $\| \cdot \|_{J^{\mu,\theta}(\mathbb{R})}$, and $\| \cdot \|_{H^{\mu}(\mathbb{R})}$, respectively. Finally, we let $H^{\mu}(a,b)$ be the fractional Sobolev space of order $\mu$ which can be defined to be the restriction of functions in $H^\mu(\mathbb{R})$ to the interval $(a,b)$, and $H^{-\mu}(a,b)$ be the dual space of $H^\mu_0(a,b)$ \cite{AdaFou}.

Below we cite some known results in the literature and prove some others that relate different fractional derivatives, spaces and (semi-)norms. We use $C$ to denote a generic constant that may assume different values at difference occurrences. We use $C_i$ to denote fixed constants. 

The following lemmas were proved in \cite{ErvRoo05,Pod,SamKil}.

\begin{lemma}\label{lem:2C1} {\rm (Fractional Poincar\'{e} inequality)} Let $\mu > 1/2 $. Then there exists a positive constant $C_0 = C_0(\mu)$ such that the following inequality holds
$$\| w \|_{L^2(a,b)} \le C_0 | w |_{H^\mu(a,b)}, \qquad \forall\ w \in H^\mu_0(a,b).$$
\end{lemma}

\begin{lemma}\label{lem:2C2} Let $\mu > 0$ and $\mu \neq m - 1/2$ with $m \in \mathbb{N}$. The spaces $J^{\mu}_{l,0}(a,b)$, $J^{\mu}_{r,0}(a,b)$, and $H^\mu_0(a,b)$ are equal with equivalent semi-norms and norms, i.e., there exist positive constants $C_1 = C_1(\mu)$ and $C_2 = C_2(\mu)$ such that
$$\begin{array}{l}
C_1 | w |_{H^\mu(a,b)} \le |w|_{J_l^{\mu}(a,b)} = |w|_{J_r^{\mu}(a,b)} \le C_2 | w |_{H^\mu(a,b)}, \\[0.1in]
C_1 \| w \|_{H^\mu(a,b)} \le \|w\|_{J_l^{\mu}(a,b)} = \|w\|_{J_r^{\mu}(a,b)} \le C_2 \| w \|_{H^\mu(a,b)}.
\end{array}$$
\end{lemma}

\begin{corollary}\label{cor:2C3}
Under the condition of Lemma \ref{lem:2C1} and $0 \leq \theta \leq 1$, we have
$$\begin{aligned}
& (C_1/2) | w |_{H^\mu(a,b)} \le |w |_{J^{\mu,\theta}(a,b)} \le C_2 | w |_{H^\mu(a,b)}, \\
& (C_1/2) \| w \|_{H^\mu(a,b)} \le \|w \|_{J^{\mu,\theta}(a,b)} \le C_2 \| w \|_{H^\mu(a,b)}.
\end{aligned}$$
\end{corollary}

\begin{lemma}\label{lem:2C4} The left and right fractional integral operators are adjoint in the $L^2$-sense, i.e., for all $\mu > 0$
$$\big({}_a^lI^{\mu}_x ~w, v \big)_{L^2(a,b)} = \big(w,{}_x^rI^{\mu}_b~ v \big)_{L^2(a,b)}, \quad \forall ~ w, v \in L^2(a,b).$$
The left and right fractional integral operators follow the properties of a semigroup, i.e., for any $w \in L^p(a,b)$ with $p \ge 1$,
$$\begin{array}{rcll}
{}_a^lI^{\mu}_x ~ {}_a^lI^{\sigma}_x w &=& {}_a^lI^{\mu+\sigma}_x w \qquad \forall x \in [a,b], ~~\forall \mu, \sigma > 0, \\[0.05in]
{}_x^rI^{\mu}_b ~{}_x^rI^{\sigma}_b w &=& {}_x^rI^{\mu+\sigma}_b w \qquad \forall x \in [a,b], ~~\forall \mu, \sigma > 0.
\end{array}$$
\end{lemma}

\begin{lemma}\label{lem:2C5} For $\mu > 0$, the following relations hold for any $w \in H^{\mu}_0(a,b)$
$$\bigl ({}_a^lD^{\mu }_xw, {}_x^rD^{\mu}_bw \bigr)_{L^2(a,b)} = \cos(\pi \mu) \|{}_a^l{D}^{\mu}_xw\|^2_{L^2(a,b)}
= \cos(\pi \mu) \|{}_x^rD^{\mu}_bw\|^2_{L^2(a,b)}.$$
\end{lemma}

\begin{lemma}\label{lem:2C7} Let $0 < \sigma <1$. Then for any $w \in W^{1,1}(a,b)$ with $w(a)= 0$
$$D~{}^{l}_aI_x^\sigma w  = {}^l_a I_x^\sigma Dw, \quad x \in (a,b),$$
and for any function $w \in W^{1,1}(a,b)$ with $w(b)= 0$
$$D~{}^{r}_{x}I^\sigma_bw = {}^r_xI^\sigma_b Dw, \quad x \in (a,b).$$
\end{lemma}
\begin{proof} By symmetry, we only prove the first equation.
$$\begin{aligned}
D ~{}^l_a I_x^{\sigma}w & = \f1{\Gamma(\sigma)} \f{d}{dx} \int_a^x \f{w(s)}{(x-s)^{1-\sigma}} ds\\
& \ds = \f1{\Gamma(1+\sigma)} \f{d}{dx} \big [ -(x-s)^{\sigma} w(s) \Big |_{s=a}^{s=x}
+ \int_a^x (x-s)^{\sigma} w'(s) ds \big ] \\
& \ds =  \f1{\Gamma(\sigma)} \int_a^x \f{w'(s)}{(x-s)^{1-\sigma}} ds = {}^{l}_{a}I_x^{\sigma} Dw.
\end{aligned}$$
This concludes the proof of the lemma. \end{proof}

\section{Previous results for constant-coefficient FDEs}

In their pioneer work \cite{ErvRoo05} Ervin and Roop studied the wellposedness of problem (\ref{Model:e1}) with a constant diffusivity coefficient $K$. They introduced a Galerkin weak formulation: For $f \in H^{-(1-\beta/2)}(a,b)$, find $u \in H^{1-\beta/2}_0(a,b)$ such that for any $v \in H^{1-\beta/2}_0(a,b)$
\begin{equation}\label{Prv:e1}
B(u,v) := \theta K \bigl \langle {}^l_aI_x^{\beta} D u, Dv \bigr \rangle + (1-\theta) K \bigl \langle {}^{r}_xI_b^{\beta}D u, Dv \bigr \rangle = \bigl \langle f, v \bigr \rangle.
\end{equation}
They proved the following theorems for the wellposedness of the Galerkin weak formulation (\ref{Prv:e1}) and its corresponding Galerkin finite element approximations \cite{ErvRoo05}.
\begin{theorem}\label{thm:Constant}
The bilinear form $B(\cdot,\cdot)$ is coercive and bounded on $H^{1-\beta/2}_0(a,b) \times H^{1-\beta/2}_0(a,b)$, so problem (\ref{Prv:e1})  is well posed. 
\end{theorem}

\begin{theorem}\label{thm:GAL1} Let $S_h^m(a,b) \subset H^{1-\beta/2}_0(a,b)$ consist of piecewise polynomials of degree up to $(m-1)$ with respect to a quasiuniform partition of diameter $h$ and $u_h \in S_h^m(a,b)$ satisfy  
\begin{equation}\label{Prv:e2}
B(u_h,v_h) = \bigl \langle f, v_h \bigr \rangle \qquad \forall v_h \in S_h^m(a,b).
\end{equation}
Assume that the weak solution $u$ to problem (\ref{Prv:e1}) is in $H^m(a,b) \cap  H^{1-\beta/2}_0(a,b)$. Then an optimal-order error estimate in the energy norm holds
\begin{equation}\label{Prv:e3}
\| u_h - u \|_{H^{1-\beta/2}} \le C h^{m-1+\beta/2} \|u\|_{H^m}.
\end{equation}
Furthermore, if the true solution $w_g$ to the dual problem of (\ref{Prv:e1}) is in $H^{2-\beta}(a,b) \cap H^1_0(a,b)$ for \underline{each} $g \in L^2(a,b)$ such that 
\begin{equation}\label{Prv:e3a}
\| w_g \|_{H^{2-\beta}} \le C \| g\|_{L^2},
\end{equation}
then an optimal-order error estimate in the $L^2$ norm holds
\begin{equation}\label{Prv:e3b}
\| u_h - u\|_{L^2} \le C h^m \|u\|_{H^m}.
\end{equation}
\end{theorem}

It was shown in \cite{JinLaz,WanYanZhu,WanYanZhu1,WanZha} that the true solution to the homogeneous Dirichlet boundary-value problem of one-dimensional steady-state FDEs (\ref{Model:e1}) of constant coefficients and right-hand side is not even in $W^{1,1/\beta}(0,1)$ for any $0 < \beta < 1$. In particular, $u \notin H^1(0,1)$ for any $1/2 \le \beta < 1$! This is in sharp contrast to the case of integer-order elliptic PDEs. To date there are no verifiable conditions on the coefficients and source terms of FDEs in the literature that can ensure the existence of smooth true solutions to FDEs. Consequently, there are no verifiable conditions to guarantee the high-order convergence rates of the numerical discretizations to FDEs. Moreover, the lack of full regularity (\ref{Prv:e3a}) of the solution to the dual FDE  also implies that any Nitsche-lifting based proof of the optimal-order $L^2$ error estimates of the form (\ref{Prv:e3b}) in the literature \cite{ErvRoo05} is invalid!

Another natural and fundamental question is as follows: Whether a variable-coefficient analogue of the bilinear form $B(\cdot,\cdot)$ in (\ref{Prv:e1}) is coercive, which in turn ensures the wellposedness of problem (\ref{Model:e1})? However, the following lemma gives rise to a negative answer to the question. 

\begin{lemma}\label{lem:Example} For any $ 0 < \beta <1$ and $ 0 \leq \theta \leq 1$, there exists a variable diffusivity coefficient $K = K(x,\beta,\theta)$ with positive lower and upper bounds and a function $w \in H^{1-\beta/2}_0(0,1)$ such that $B(w,w) < 0$. In fact, $K$ can be chosen as piecewise constant with just three pieces or its smooth modification.
\end{lemma}

\begin{proof} We prove the lemma by construction. Choose $w$ to be the following continuous and
piecewise-linear function
$$w(x) :=\left\{ \begin{array}{ll}
4x, & \ds x \in \Big[0,\f1{4} \Big],\\[0.1in]
\ds 4\Big(\f1{2} - x \Big), & \ds x \in \Big [\f1{4},\f3{4} \Big],\\[0.1in]
-4(1-x), \quad & \ds x \in \Big [ \f3{4},1 \Big].
\end{array}\right.$$
Apparently, $w \in H^1_0(0,1)$. Direct calculation yields
$${}_0^l I_x^{\beta} D w(x) =\f4{\Gamma(\beta+1)} \left\{ \begin{array}{ll}
\ds  x^\beta, & \ds x \in \Bigl [0,\f1{4}\Bigr],\\[6pt]
\ds  x^\beta-2\Bigl(x-\f1{4}\Bigr)^{\beta}, & \ds x \in \Bigl[\f1{4},\f3{4}\Bigr],\\[6pt]
\ds  x^\beta-2\Bigl(x-\f1{4}\Bigr)^{\beta}+2\Bigl(x-\f3{4}\Bigr)^{\beta}, ~~& \ds x \in \Bigl[\f3{4},1\Bigr]
\end{array}\right.$$
and
$$\begin{array}{l}
\ds {}_x^rI_1^{\beta} Dw(x) = \f4{\Gamma(\beta+1)}\left\{ \begin{array}{ll}
\ds  (1-x)^\beta-2\Bigl(\f3{4}-x\Bigr)^\beta+2\Bigl(\f1{4}-x\Bigr)^\beta, & \ds x \in \Bigl[0,\f1{4}\Bigr],\\[6pt]
\ds  (1-x)^\beta-2\Bigl(\f3{4}-x\Bigr)^{\beta}, & \ds x \in \Bigl[\f1{4},\f3{4}\Bigr],\\[6pt]
\ds  (1-x)^\beta, ~~ &\ds x \in \Bigl[\f3{4},1\Bigr].
\end{array}\right.
\end{array}$$
We now prove that there exists a variable diffusivity coefficient such that $B(w,w)<0$. In fact, direct calculation shows
\begin{equation}\label{Exam:e1}\begin{aligned}
&\big(\theta {}_0^l I_x^{\beta} + (1-\theta) {}_x^r I_1^{\beta}\big) D w(x) \bigl |_{x=\f1{4}} \\
&\quad = \frac{4^{1-\beta}}{\Gamma(\beta+1)} \big[(2\theta-1) + (1-\theta) (1+3^\beta-2^{1+\beta})\bigr]
\end{aligned}\end{equation}
and
\begin{equation}\label{Exam:e2}\begin{aligned}
&\big(\theta {}_0^lI_x^{\beta} + (1-\theta){}_x^rI_1^{\beta}\big) D w(x) \bigl |_{x=\f3{4}} \\
& \quad = \frac{4^{1-\beta}}{\Gamma(\beta+1)} \bigl [(1-2\theta) + \theta(1+3^\beta-2^{1+\beta})\bigr].
\end{aligned}\end{equation}
It is easy to check that $1+3^\beta-2^{1+\beta} \leq 0$ for $ 0 \leq \beta \leq 1$. Moreover,  $1+3^\beta-2^{1+\beta}=0$ if and only if $\beta=0$ or $\beta=1$. Hence, in the current context of $0 < \beta < 1$,  $\lambda = \lambda(\beta) = 2^{1+\beta}-1-3^\beta >0$. 

We now consider the case $0 \leq \theta \leq 1/2$. We observe from (\ref{Exam:e1} that
$$\begin{aligned}
-\frac{4^{1-\beta}(1+\lambda)}{\Gamma(\beta+1)} &\le \big(\theta {}_0^l I_x^{\beta} + (1-\theta) {}_x^r I_1^{\beta}\big) D w(x) \bigl |_{x=\f1{4}} \\
&= \frac{4^{1-\beta}}{\Gamma(\beta+1)} \big[(2\theta-1) -(1-\theta)\lambda \bigr] \le -\frac{4^{1-\beta}\lambda}{2\Gamma(\beta+1)}. \end{aligned}$$
By the continuity of ${}_0^lI_x^{\beta} D w(x)$ and ${}_x^r I_1^{\beta}D w(x)$ there exists a $0 < \delta \leq 1/4$ such that
$$-\frac{2\cdot4^{1-\beta}(1+\lambda)}{\Gamma(\beta+1)} \leq \bigl(\theta {}_0^l{D}_x^{-\beta} + (1-\theta) {}_x^r D_1^{1-\beta} \big) D w
\leq -\frac{4^{1-\beta}\lambda}{4\Gamma(\beta+1)}, \quad  x \in \Bigl[\f1{4}-\delta,\f1{4}\Bigr].$$
We accordingly define a diffusivity coefficient $K(x)$ as follows
$$K(x) :=\left\{ \begin{array}{ll}
K_l, \quad & \ds x \in \Bigl(0,\f{1}{4}-\delta\Bigr),\\[0.05in]
1, \quad & \ds x \in \Bigl(\f1{4}-\delta,\f{1}{4}\Bigr),\\[0.1in]
K_r, & \ds x \in \Bigl(\f{1}{4},1\Bigr),
\end{array}\right.$$
where $K_l$ and $K_r$ are positive constants to be determined. Then we have
$$\begin{aligned}
B(w,w)= & 4 K_l\int^{\f1{4}-\delta}_0(\theta {}_0^lI_x^{\beta} + (1-\theta) {}_x^r I_1^{\beta})D wdx + 4 \int^{\f1{4}}_{\f1{4}-\delta}(\theta {}_0^lI_x^{\beta} + (1-\theta){}_x^rI_1^{\beta})D wdx \\
& -4 K_r\int^{\f3{4}}_{\f1{4}}(\theta {}_0^lI_x^{\beta} + (1-\theta) {}_x^r I_1^{\beta}) D wdx + 4 K_r\int^1_{\f3{4}}(\theta {}_0^lI_x^{\beta} + (1-\theta){}_x^rI_1^{\beta})D wdx \\
\le & -\frac{4^{1-\beta}\lambda\delta}{\Gamma(\beta+1)} + 4 K_l \int^{\f1{4}-\delta}_0 \bigl | (\theta {}_0^lI_x^{\beta}+ (1-\theta){}_x^rI_1^{\beta})D w \bigr|dx \\
& + 4 K_r\int^1_{\f1{4}} \bigl |(\theta {}_0^lI_x^{\beta} +(1-\theta){}_x^rI_1^{\beta})D w \bigr|dx.
\end{aligned}$$
We note that the integrands in the two integrals on the right-hand side are uniformly bounded from above with respect to $\theta \in [0,1/2]$. Hence, by choosing the positive constants $K_l$ and $K_r$ sufficiently small we can enforce $B(w,w) < 0$. We can similarly prove the conclusion for the case of $1/2 \leq \theta \leq 1$ by using (\ref{Exam:e2}).

We also observe from the proof that we can connect the piecewise constant diffusivity coefficient $K(x)$ as smooth as one desired such that the bilinear form $B(w,w)$ still loses its coercivity at least for some $w \in H^1_0(0,1)$.
\end{proof}

\begin{remark}
We observe from the proof of Lemma \ref{lem:Example} that the fundamental reason for the bilinear form $B(w,w)$ to lose its coercivity is that $Dw$ and $(\theta {}_0^l{I}_x^{\beta} + (1-\theta) {}_x^rI_1^{\beta}) Dw$ do not always retain the same sign for all the functions $w \in H^1_0(0,1)$. As long as there exists one function $w$ and a subinterval on which $(\theta {}_0^l{I}_x^{\beta} +(1-\theta){}_x^rI_1^{\beta}) Dw) ~Dw < 0$, one can always enforce $B(w,w) < 0$ by choosing a specific diffusivity coefficient $K$ appropriately. Finally, a careful examination of the counterexample shows that the bilinear form $B(w,w)$ with a variable diffusivity coefficient having large variations might lose its coercivity.
\end{remark}

\section{A two-sided fractional integral operator $I^{\beta}_\theta$ and its properties}

Besides the pioneer work \cite{ErvRoo05} of Ervin and Roop on the well-posedness of the two-sided FDE (\ref{Model:e1}) with a constant diffusivity coefficient $K$, virtually almost all the rest of the well-posedness results were proved for the FDE (\ref{Model:e1}) are only for a one-sided simplification of problem (\ref{Model:e1}) with either a constant diffusivity coefficient $K$ \cite{JinLaz} or a variable diffusivity coefficient $K$ \cite{WanYan,WanYanZhu}. To the best of our knowledge, there is no well-posedness result on problem (\ref{Model:e1}) with a variable diffusivity coefficient $K$. Moreover, there is no regularity result in the literature for the two-sided problem (\ref{Model:e1}) even for a constant diffusivity coefficient $K$.

To study the well-posedness and some regularity result of the two-sided problem (\ref{Model:e1}) we introduce the following two-sided fractional integral operator $I^{\beta}_\theta$ for $ 0 < \beta <1$ and $ 0 \leq \theta \leq 1$:
\begin{equation}\label{Int:e0} I^{\beta}_\theta w = \theta~{}^{l}_aI^{\beta}_xw+(1-\theta)~{}^{r}_xI^{\beta}_bw.\end{equation}
We note that in the case of $\theta = $ 0 or 1 
the two-sided integral operator $I^{\beta}_\theta$ reduces to the well known Volterra integral operators, which have been well studied \cite{Kre,SamKil}. However, $I^{\beta}_\theta$ is a convex combination of two Volterra integral operators for $0 < \theta < 1$, for which there seems to be little study in the literature. We study its properties in the following theorem.

Without loss of generality from now on we assume $a = 0$ and $b = 1$ for simplicity of presentation. We apply Lemma \ref{lem:2C7} to express the boundary-value problem (\ref{Model:e1}) in terms of the fractional integral operator $I^{\beta}_\theta$ as follows
\begin{equation}\label{Int:e1}\begin{array}{rcl}
\ds - D \bigl ( KD I^{\beta}_\theta u \bigr ) &=& f(x), \quad  x \in (0,1), \\[0.075in]
u(0) = u(1) &=& 0.
\end{array}\end{equation}
In other word, the two-sided variable-coefficient FDE in problem (\ref{Model:e1}) can naturally be rewritten as a canonical second-order diffusion equation in terms of $I^{\beta}_\theta u$. However, the homogeneous Dirichlet boundary condition in terms of $u$ cannot be naturally expressed in terms of $I^{\beta}_\theta u$, and in fact becomes one of the fundamental difficulties to overcome in the study of problem (\ref{Model:e1}).

\begin{lemma}\label{lem:dbeta} Let $ 0 < \beta <1/2$ and $0 \le \theta \le 1$. Then for any $w \in H^{1-\beta}_0(0,1)$ we have
$$\begin{array}{rl}
\ds \f{1-\cos(\pi \beta)}2 |w|_{J^{1-\beta,\theta}(0,1)}^2 & \ds \le \big(1-\cos(\pi \beta)\big) \big(\theta^2 + (1-\theta)^2 \big) |w|_{J^{1-\beta,\theta}(0,1)}^2 \\[0.05in]
& \ds \le \big \| D I^{\beta}_\theta w \big \|_{L^2(0,1)}^2 = \big \| I^{\beta}_\theta D w \big \|_{L^2(0,1)}^2 \\ [0.1in]
& \ds \le \big(1 + \cos(\pi \beta)\big) \big(\theta^2 + (1-\theta)^2 \big) |w|_{J^{1-\beta,\theta}(0,1)}^2\\[0.1in]
& \ds \le \big(1 + \cos(\pi \beta)\big)  |w|_{J^{1-\beta,\theta}(0,1)}^2
\end{array}$$
where $| \cdot |_{J^{1-\beta,\theta}(0,1)}$ is defined below (\ref{FDE:e2}). \end{lemma}
\begin{proof} We use the definition of $I^{\beta}_\theta w$ to obtain
$$\begin{aligned}
& \big \|D I^{\beta}_\theta w \big \|_{L^2(0,1)}^2 \\
& \ds  = \bigl ( \theta~{}_0^lD^{1-\beta}_x w - (1-\theta)~{}_x^rD^{1-\beta}_1 w,
\theta~{}_0^lD^{1-\beta}_x w - (1-\theta)~{}_x^rD^{1-\beta}_1 w \bigr)_{L^2(0,1)}\\
& \quad = \theta^2 \|{}_0^lD^{1-\beta}_x w \|_{L^2(0,1)}^2 + (1-\theta)^2
\| {}_x^rD^{1-\beta}_1 w \|_{L^2(0,1)}^2\\[0.025in]
&\ds \qquad -2\theta(1-\theta) \bigl ({}_0^lD^{1-\beta}_xw, {}_x^rD^{1-\beta}_1w \bigr)_{L^2(0,1)}
\\[0.025in]
&\quad = \theta^2 \|{}_0^lD^{1-\beta}_x w \|_{L^2(0,1)}^2 + (1-\theta)^2
\| {}_0^l D^{1-\beta}_x w \|_{L^2(0,1)}^2\\[0.025in]
&\ds \qquad +2\theta(1-\theta) \cos(\pi \beta)  \|{}_0^lD^{1-\beta}_x w \|_{L^2(0,1)}^2 \\[0.025in]
&\quad \ge \theta^2 \|{}_0^lD^{1-\beta}_x w \|_{L^2(0,1)}^2 + (1-\theta)^2
\| {}_0^l D^{1-\beta}_x w \|_{L^2(0,1)}^2\\[0.025in]
&\ds \qquad -(\theta^2 + (1-\theta)^2) \cos(\pi \beta)  \|{}_0^lD^{1-\beta}_x w \|_{L^2(0,1)}^2 \\[0.025in]
&\quad = (1 - \cos(\pi \beta)) |w|_{J^{1-\beta,\theta}(0,1)}^2.
\end{aligned}$$
The other inequality can be proved similarly. \end{proof}

We are now in the position to study the properties of $I^{\beta}_\theta w$. 

\begin{theorem}\label{thm:Range} Let $0 < \beta < 1/2$ and $0 \le \theta \le 1$. The fractional integral operator $I^{\beta}_\theta$ is a bounded linear bijection from $H^{1-\beta}_0(0,1)$ to its closed range $\mathcal{R}(I^{\beta}_\theta) \subset H^1(0,1)$ with
$$\mathcal{R}(I^{\beta}_\theta) :=\big \{ w \in H^1(0,1): \exists \phi \in H^{1-\beta}_0(0,1),~~ s.t. ~~ w = I^{\beta}_\theta \phi \big\}.$$
Moreover, the inverse operator $(I^{\beta}_\theta)^{-1}$ of $I^{\beta}_\theta$ is also bounded.
\end{theorem}

\begin{proof} We prove the theorem in two steps. At step 1 we prove that $I^{\beta}_\theta$ is a bounded linear operator from $H^{1-\beta}_0(0,1)$ to $H^1(0,1)$. It is clear that the integral $I^{\beta}_\theta \phi$ is well defined for any $\phi \in H^{1-\beta}(0,1) \subset C[0,1]$. Then we apply Lemma \ref{lem:2C2} to obtain
$$\begin{aligned}
\big \|DI^{\beta}_\theta \phi \big \|_{L^2(0,1)} &\leq \theta\|{}_0^lD^{1-\beta}_x \phi \|_{L^2(0,1)}+(1-\theta)\|{}_x^rD^{1-\beta}_1 \phi \|_{L^2(0,1)}\\
&\leq C| \phi |_{H^{1-\beta}(0,1)}, \quad \forall \ \phi \in H^{1-\beta}_0(0,1).
\end{aligned}$$
Furthermore, we have
$$\begin{aligned}
\big | {}_0^lI^{\beta}_x \phi \big| & = \Bigl |\int^x_0D({}_0^lI^{\beta}_s \phi) ds \Bigr | \leq \sqrt{x} \big \|{}_0^lD^{1-\beta}_x \phi \big \|_{L^2(0,1)},\\
\big | {}_x^rI_1^{\beta} \phi \big| & = \Bigl |\int^1_x D({}_s^rI_1^{\beta} \phi) ds \Bigr | \leq \sqrt{1-x} \big \|{}_x^r D^{1-\beta}_1\phi \big \|^2_{L^2(0,1)}.
\end{aligned}$$
Hence,
$$\big \| I^{\beta}_\theta \phi \big \|_{L^2(0,1)} \leq C | \phi |_{H^{1-\beta}(0,1)}, \quad \forall\ \phi \in H^{1-\beta}_0(0,1).$$
We combine the preceding estimates to finish the proof of step 1. 

At step 2 we prove that $I^{\beta}_\theta$ has a bounded inverse operator $(I^{\beta}_\theta)^{-1}$ from $\mathcal{R}(I^{\beta}_\theta)$ onto $H^{1-\beta}_0(0,1)$ and that $\mathcal{R}(I^{\beta}_\theta)$ is a closed subspace of $H^1(0,1)$. In fact, we apply Lemmas \ref{lem:2C2} and \ref{lem:dbeta} to conclude that for any $\phi \in H^{1-\beta}_0(0,1)$
$$\big \| I^{\beta}_\theta \phi \big \|^2_{H^1(0,1)} \ge \big \|D I^{\beta}_\theta \phi \big \|^2_{L^2(0,1)}
\ge \f{1-\cos(\pi \beta)}2 | \phi |_{J^{1-\beta,\theta}(0,1)}^2 \ge \eta \| \phi \|^2_{H^{1-\beta}(0,1)} $$
for some $\eta = \eta(\beta) >0$. Hence the operator $I^{\beta}_\theta$ is invertible and its inverse operator is bounded by $1/\eta$. Further, since both $I^{\beta}_\theta$ and its inverse operator are bounded linear operators, $\mathcal{R}(I^{\beta}_\theta)$ is a closed subspace of $H^1(0,1)$. \end{proof}

To further study the properties of the range $\mathcal{R}(I^{\beta}_\theta)$, let $P$ be a projection operator from $H^1(0,1)$ onto $H^1_0(0,1)$ defined as follows: for any $w \in H^1(0,1)$, seek $Pw \in H^1_0(0,1)$ such that
\begin{equation}\label{Int:e2}
(D Pw, D v)_{L^2(0,1)} = (D w, D v)_{L^2(0,1)}, \quad \forall\ v \in H^1_0(0,1).
\end{equation}
Let $\mathcal{N}(P)$ be the null space of the operator $P$
\begin{equation}\label{Int:e3}\begin{array}{rl}
\mathcal{N}(P) := & \bigl \{w \in H^1(0,1): \ Pw=0 \bigr \}\\[0.05in]
=& \bigl \{ w \in H^1(0,1): \ (D w, D v)_{L^2(0,1)} =0, \ \ \forall\ v \in H^1_0(0,1) \bigr\}.
\end{array}
\end{equation}
Choosing any $w \in H^1_0(0,1)$ concludes immediately that $H^1_0(0,1) = P(H^1(0,1))$. Hence,  the following decomposition
$$H^1(0,1) = H^1_0(0,1) \oplus \mathcal{N}(P)$$
holds. Since $\mathcal{R}(I^{\beta}_\theta) \subset H^1(0,1)$, $P(\mathcal{R}(I^{\beta}_\theta)) \subset H^1_0(0,1)$. In the next theorem we prove that the equality actually holds.

\begin{theorem}\label{thm:Decomp} Let $0 < \beta < 1/2$ and $0 \le \theta \le 1$. The following equalities hold
\begin{equation}\label{Int:e4} \mathcal{N}(P) = span\{w_l^c, w_r^c\}, \quad P\bigl(\mathcal{R}\bigl(I^{\beta}_\theta \bigr)\bigr) = H^1_0(0,1), \quad H^1(0,1) =\mathcal{R}\bigl(I^{\beta}_\theta \bigr) \oplus \mathcal{N}(P)
\end{equation}
where $w_l^c := 1-x$ and $w_r^c: = x$. 
\end{theorem}

\begin{proof} We prove the theorem in four steps. As the first step we prove the first equality in (\ref{Int:e4}). Note that $w_l^c$ and $w_r^c$ in $H^1(0,1)$ satisfy
\begin{equation}\label{Int:e5}\begin{aligned}
& (Dw_l^c,Dv)_{L^2(0,1)} = 0, \ \ \forall \ v \in H^1_0(0,1); \quad w_l^c(0)=1, ~w_l^c(1)=0;\\[0.025in]
& (Dw_r^c,Dv)_{L^2(0,1)} = 0, \ \ \forall \ v \in H^1_0(0,1); \quad w_r^c(0)=0, ~w_r^c(1)=1.
\end{aligned}
\end{equation}
Hence, $w_l^c, w_r^c \in \mathcal{N}(P)$. Conversely, for any $w \in \mathcal{N}(P)$, it is clear that $w - w(0)w_l^c - w(1)w_r^c \in H^1_0(0,1) \cap \mathcal{N}(P)$ as
$$\bigl(D(w - w(0)w_l^c - w(1)w_r^c), Dv \bigr)_{L^2(0,1)} = 0, \ \ \forall \ v \in H^1_0(0,1).$$
Choosing $v = w - w(0)w_l^c - w(1)w_r^c$ in this equation concludes that $w = w(0)w_l^c + w(1)w_r^c$. Thus, the first equality in (\ref{Int:e4}) holds. 

As the next step, we prove that the space
\begin{equation}\label{Int:e5a}
\mathcal{U} := \bigl \{ w \in H^1_0(0,1): \exists\ \phi \in H^{1-\beta}_0(0,1)~\mathrm{and}~c, c' \in \mathbb{R} ~\mathrm{s.t.} ~I^{\beta}_\theta \phi = w + c w_l^c + c' w_r^c \bigr \}\end{equation}
is a closed subspace of $H^1_0(0,1)$.

First, it is clear that $\mathcal{U}$ is a subspace of $H^1_0(0,1)$. Let $w_1, w_2 \in \mathcal{U}$ and $a_1, a_2 \in \mathbb{R}$, then there exist $\phi_1, \phi_2 \in H^{1-\beta}_0(0,1)$ such that 
$$I^{\beta}_\theta (a_1 \phi_1 + a_2 \phi_2) = (a_1w_1 + a_2w_2) + (a_1 c_1 + a_2 c_2) w_l^c + (a_1 c'_1 + a_2 c'_2) w_r^c.$$
Hence, $a_1 w_1 + a_2 w_2 \in \mathcal{U}$. We now prove that $\mathcal{U}$ is closed. To do so, let $\{w_n\}^\infty_{n=1} \subset \mathcal{U}$ be a sequence that converges to $w \in H^1_0(0,1)$. By definition of $\mathcal{U}$, there exist sequences $\{ \phi_n\}_{n=1}^\infty \subset H^{1-\beta}_0(0,1)$ and $\{c_n\}_{n=1}^\infty$, $\{c'_n\}_{n=1}^\infty \subset \mathbb{R}$ such that
\begin{equation}\label{Int:e6}
I^{\beta}_\theta \phi_n = w_n + c_n w_l^c + c'_n w_r^c \in \mathcal{R}(I^{\beta}_\theta), \quad n \ge 1.
\end{equation}

As $\{\|w_n\|_{H^1(0,1)}\}^\infty_{n=1}$ is bounded, we claim that both sequences $\{c_n\}_{n=1}^\infty$ and $\{c'_n\}_{n=1}^\infty$ are bounded. Otherwise, there exists a subsequence $\{n_k\}_{k=1}^\infty$ such that $\lim_{k \rightarrow \infty} \max\{|c_{n_k}|, |c'_{n_k}|\}=+\infty$. Without loss of generality, we assume that $|c_{n_k}| = \max\{|c_{n_k}|, |c'_{n_k}|\}$ so $\lim_{n_k \rightarrow \infty} |c_{n_k}| = \infty$. Since $|c'_{n_k}|/|c_{n_k}| \leq 1$, there exists a subsequence which we still denote by $\{c'_{n_k}/c_{n_k}\}$ such that $\lim_{k \rightarrow \infty }c'_{n_k}/c_{n_k} = c'$. Let $\psi_{n_k} = \phi_{n_k}/c_{n_k}$. Then $\psi_{n_k} \in H^{1-\beta}_0(0,1)$. We have
$$I^{\beta}_\theta\psi_{n_k}= \frac{1}{c_{n_k}}\big(w_{n_k}+c_{n_k}w_l^c+c'_{n_k}w_r^c \big) \rightarrow w_l^c + c' w_r^c, \ \  \ as \  \ n_k \rightarrow \infty.$$
As $\mathcal{R}(I^{\beta}_\theta)$ is closed, $w_l^c + c' w_r^c \in \mathcal{R}(I^{\beta}_\theta)$. By Theorem \ref{thm:Range}, there exists a $\psi \in H^{1-\beta}_0(0,1)$ such that $I^{\beta}_\theta \psi = w_l^c + c' w_r^c$.
Consequently,
\begin{equation}\label{Int:e6a}
\bigl (D(I^{\beta}_\theta\psi),Dv \bigr)_{L^2(0,1)} = \bigl (D (w_l^c + c' w_r^c),Dv \bigr)_{L^2(0,1)} =0, \ \ \forall\ v \in H^1_0(0,1).
\end{equation}
Since this problem apparently has a trivial solution, the uniqueness of the solution of the problem ensured by Lemma \ref{lem:2C7} and Theorem \ref{thm:Constant} concludes that $\psi \equiv 0$. That is, $0 = w_l^c + c' w_r^c$. This contradicts to the linear independence of $w_l^c$ and $w_r^c$. We thus have proved that $\max\{|c_n|,|c'_n|\}$ is bounded.

Consequently, there exist convergent subsequences $(c_{n_k},c'_{n_k})$ that converge to $(c,c')$ as $k \rightarrow \infty$. We pass the limit in (\ref{Int:e6}) to the subsequence to deduce that there exists a $\phi \in H^{1-\beta}_0(0,1)$ such that
$$\lim_{k \rightarrow \infty}\phi_n = \lim_{k \rightarrow \infty}(I^{\beta}_\theta)^{-1}(w_{n_k}+c_{n_k} w_l^c + c'_{n_k} w_r^c) = (I^{\beta}_\theta)^{-1}(w +c w_l^c + c' w_r^c) = \phi.$$
That is, $I^{\beta}_\theta \phi = w + c w_l^c + c' w_r^c$ which implies that $w \in \mathcal{U}$. We have thus proved that $\mathcal{U}$ is closed.

At the third step, we prove that $\mathcal{U} = H^1_0(0,1)$. In fact, for any $g \in H^1_0(0,1) \bigcap H^2(0,1)$, Lemma \ref{lem:2C7} and Theorem \ref{thm:Constant} ensure that the problem
$$ \big(D I^{\beta} \phi,Dv \big)_{L^2(0,1)} = - \big(D^2 g, v \big)_{L^2(0,1)}, \ \ \forall\ v \in H^{1-\frac{\beta}{2}}_0(0,1)$$
has a unique solution $\phi \in H^{1-\frac{\beta}{2}}_0(0,1)$. This equation can then be rewritten as
$$\bigl (D(I^{\beta}\phi - g),Dv \bigr)_{L^2(0,1)} = 0, \quad \forall\ v \in H^1_0(0,1).$$
This implies that $I^{\beta}\phi-g \in \mathcal{N}(P)$. Hence, there exist constants $c$ and $c'$ such that $I^{\beta}\phi = g + c w_l^c + c' w_r^c$.
This shows that $g \in \mathcal{U}$ for any $g \in H^1_0(0,1) \cap H^2(0,1)$. That is, $H^1_0(0,1) \cap H^2(0,1) \subset \mathcal{U} \subset H^1_0(0,1)$. Since $H^1_0(0,1) \cap H^2(0,1)$ is dense in $H^1_0(0,1)$ and $\mathcal{U}$ is closed, we conclude that $\mathcal{U} = H^1_0(0,1)$.

Finally, at step 4 we prove the last equality in (\ref{Int:e4}). For any $w \in H^1(0,1)$, it is clear that $v = w - w(0)w_l^c - w(1)w_r^c \in H^1_0(0,1) = \mathcal{U}$. By (\ref{Int:e5a}), there exist $\phi \in H^{1-\beta}_0(0,1)$ and $c, c' \in \mathbb{R}$ such that $I^{\beta}_\theta \phi = v + c w_l^c + c' w_r^c$.
That is, 
$$w = I^{\beta}_\theta \phi + (w(0)-c)w_l^c + (w(1)-c') w_r^c.$$
We thus prove $H^1(0,1) = \mathcal{R}(I^{\beta}_\theta) + \mathcal{N}(P)$. We now prove $\mathcal{R}(I^{\beta}_\theta) \cap \mathcal{N}(P) = \emptyset$. 
For any $w \in \mathcal{R}(I^{\beta}_\theta) \cap \mathcal{N}(P)$, there exists $\psi \in H^{1-\beta}_0(0,1)$ and $c, c' \in \mathbb{R}$ such that 
$$w = I^{\beta}_\theta \psi = c w_l^c + c' w_r^c.$$ 
Then the same argument following (\ref{Int:e6a}) shows that $\psi \equiv 0$ and $c = c' = 0$. That is, $w \equiv 0$. We thus prove the third equality in (\ref{Int:e4}).
\end{proof}

We now prove the main result of this section.
\begin{theorem}\label{thm:IntOper} Let $0 < \beta < 1/2$ and $0 \le \theta \le 1$. Let $w_l^c$ and $w_r^c$ be defined as in Theorem \ref{thm:Decomp}. Then for any $g \in H^1(0,1)$, the following integral equation
\begin{equation}\label{Int6:e1}
I^{\beta}_\theta \phi + \bigl (g(0) - (1-\theta){}_0^rI^{\beta}_1\phi \big)w_l^c + \big (g(1) - \theta({}_0^lI^{\beta}_1 \phi) \big )w_r^c = g
\end{equation}
has a unique solution $\phi \in H^{1-\beta}_0(0,1)$. Conversely, for any $\phi \in H^{1-\beta}_0(0,1)$
\begin{equation}\label{Int6:e1a}
g := I^{\beta}_\theta \phi - (1-\theta)\big({}_0^rI^{\beta}_1\phi\big) w_l^c - \theta\big({}_0^lI^{\beta}_1 \phi\big) w_r^c \in H^1_0(0,1).
\end{equation}
Furthermore, if $g \in H^1_0(0,1)$, then there exist positive constants $C_3$ and $C_4$ such that
\begin{equation}\label{Int6:e2}
C_3 \| \phi \|_{H^{1-\beta}(0,1)} \le \|g\|_{H^1(0,1)} \le C_4 \|\phi \|_{H^{1-\beta}(0,1)}.
\end{equation}
\end{theorem}

\begin{proof} For any given $g \in H^1(0,1)$, by (\ref{Int:e4}) in Theorem \ref{thm:Decomp}, there exist unique $\phi \in H^{1-\beta}_0(0,1)$ and constants $c$ and $c'$ such that
\begin{equation}\label{Int6:e3}
g = I^{\beta}_\theta \phi + c w_l^c + c'w_r^c.
\end{equation}
We use the definition (\ref{Int:e0}) of $I^{\beta}_\theta$ and note that ${}_0^lI^{\beta}_0 \phi = 0$ and ${}_1^rI^{\beta}_1 \phi = 0$ to obtain
$$ g(0) = (1-\theta) {}^r_0I^{\beta}_1 \phi  + c, \qquad g(1) = \theta {}^l_0I^{\beta}_1 \phi + c'. $$
We thus prove (\ref{Int6:e1}). By Theorem \ref{thm:Range}, $g$ defined in (\ref{Int6:e1a}) is in $H^1(0,1)$. It is clear that $g(0) = g(1) = 0$. This concludes the proof of (\ref{Int6:e1a}). 

To prove (\ref{Int6:e2}) we note that
$${}_0^rI^{\beta}_1 \phi = - \int^1_0 D({}_s^rI^{\beta}_1\phi)ds, \qquad {}_0^lI^{\beta}_1 \phi = \int^1_0 D({}_0^lI^{\beta}_s\phi)ds.$$
Hence
$$\theta \bigl |{}_0^lI^{\beta}_1 \phi \bigr| + (1-\theta) \bigl |{}_0^r I^{\beta}_1 \phi \bigr| \le C | \phi |_{J^{1-\beta,\theta}(0,1)}.$$
We combine this inequality with (\ref{Int6:e1a}) and Corollary \ref{cor:2C3} to arrive at
$$\begin{aligned}
\|g\|_{H^1(0,1)} &\leq C \big(\|D I^{\beta}_\theta\phi\|_{L^2(0,1)} + \theta \bigl |{}_0^l{D}^{1-\beta}_1 \phi \bigr| + (1-\theta) \bigl |{}_0^r D^{1-\beta}_1 \phi \bigr | \big)\\[0.025in]
& \leq C|\phi|_{J^{1-\beta,\theta}(0,1)} \leq C \|\phi \|_{H^{1-\beta}(0,1)}.
\end{aligned}$$
We thus prove the right inequality in (\ref{Int6:e2}).

To prove the left inequality in (\ref{Int6:e2}), let $P$ be the projection operator from $H^1(0,1)$ onto $H^1_0(0,1)$ as defined in (\ref{Int:e2}). Choosing $v = Pw$ in (\ref{Int:e2}) yields 
\begin{equation}\label{Int6:e4}
\| D(Pw) \|_{L^2(0,1)} \leq \|Dw \|_{L^2(0,1)}, \quad \forall w \in H^1(0,1).
\end{equation}
Although $P$ is not one-to-one from $H^1(0,1)$ to $H^1_0(0,1)$, we claim that $PI^{\beta}_\theta$ is a bounded linear bijection from $H^{1-\beta}_0(0,1)$ onto $H^1_0(0,1)$. As a matter of  fact, by Theorems \ref{thm:Range} and \ref{thm:Decomp} and equation (\ref{Int6:e4}), we conclude that $PI^{\beta}_\theta$ is a bounded linear operator from $H^{1-\beta}_0(0,1)$
onto $H^1_0(0,1)$. It remains to prove that $PI^{\beta}_\theta$ is injective. Suppose not, then there exists a $0 \neq \phi \in H^{1-\beta}_0(0,1)$ such that $PI^{\beta}_\theta\phi = 0$. Thus, 
$$0 = (D(PI^{\beta}_\theta\phi),Dv)_{L^2(0,1)} = (DI^{\beta}_\theta\phi,Dv)_{L^2(0,1)},
\qquad \forall v \in H^1_0(0,1).$$
By Theorem \ref{thm:Constant}, this problem has only the trivial solution $\phi = 0$. This concludes that $PI^{\beta}_\theta$ is injective and so is a bounded linear bijection from $H^{1-\beta}_0(0,1)$ onto $H^1_0(0,1)$. By Banach's bounded inverse theorem, $PI^{\beta}_\theta$ has a bounded inverse operator $(PI^{\beta}_\theta)^{-1}$ from $H^1_0(0,1)$ onto $H^{1-\beta}_0(0,1)$. Thus, for any $g \in H^1_0(0,1)$, there exists a unique $\phi \in H^{1-\beta}_0(0,1)$ such that
$$PI^{\beta}_\theta \phi = g$$
and vice versa. Furthermore,
$$\| \phi \|_{H^{1-\beta}(0,1)} \leq C \|g\|_{H^1(0,1)}.$$
We thus prove the left inequality in (\ref{Int6:e2}) and so the theorem.
\end{proof}

\section{A Petrov-Galerkin formulation for constant-coefficient problems}

Lemma \ref{lem:Example} shows that the Galerkin formulation (\ref{Prv:e1}) may lose its coercivity in the context of variable-coefficient FDEs. Numerical evidence also indicated the illposedness of the Galerkin formulation \cite{WanYanZhu}. We note that the governing equation (\ref{Model:e1}) is obtained by incorporating the fractional Fick's law into a canonical conservation law. This motivates the following Petrov-Galerkin weak formulation for problem (\ref{Model:e1}) with $0 < \beta < 1/2$: Given $f \in H^{-1}(0,1)$, seek $u \in H^{1-\beta}_0(0,1)$ such that
\begin{equation}\label{PG:e1}
A(u,v) := \bigl (KDI^{\beta}_\theta u,Dv \big )_{L^2(0,1)} = \langle f,v \rangle, \qquad \forall v \in H^1_0(0,1).
\end{equation}

In this section we study the wellposedness of the weak formulation (\ref{PG:e1}) for problem (\ref{Model:e1}) with $K \equiv 1$ and the characterization of its solution.
\begin{theorem}\label{thm:PGC} For problem (\ref{Model:e1}) with $0 < \beta < 1/2$,  $0 \leq \theta \leq 1$, and $K \equiv 1$, the bilinear form $A(\cdot,\cdot)$ is bounded and weakly coercive on the space $H^{1-\beta}_0(0,1) \times H^1_0(0,1)$
\begin{equation}\label{PGC:e1}\begin{array}{l}
\ds \inf_{w \in H^{1-\beta}_0(0,1)\setminus\{0\}} \sup_{v \in H^1_0(0,1)\setminus\{0\}}\frac{A(w,v)}{\|w\|_{H^{1-\beta}(0,1)}\|v\|_{H^1(0,1)}} \ge \kappa,\\[0.2in]
\ds \sup_{w \in H^{1-\beta}_0(0,1)} A(w,v) > 0, \qquad \forall\ v \in H^1_0(0,1)\setminus\{0\}
\end{array}\end{equation}
for a positive constant $\kappa=\kappa(\beta) >0$. Hence, the Petrov-Galerkin formulation (\ref{PG:e1}) has a unique solution $u \in H^{1-\beta}_0(0,1)$ for which 
\begin{equation}\label{PGC:e2}
\|u\|_{H^{1-\beta}(0,1)} \leq \frac{1}{\kappa}\|f\|_{H^{-1}(0,1)}.
\end{equation}
\end{theorem}

\begin{proof} For each $w \in H^{1-\beta}_0(0,1)$, by Theorem \ref{thm:IntOper} there exists a unique $v \in H^1_0(0,1)$ such that
\begin{equation}\label{PGC:e3}
I^{\beta}_\theta w - (1-\theta)({}_0^rI^{\beta}_1 w)w_l^c - \theta({}_0^lI^{\beta}_1w)w_r^c = v
\end{equation}
where $w_l^c$ and $w_r^c$ are defined in Theorem \ref{thm:Decomp}. We differentiate this equation and apply (\ref{Int6:e2}) to obtain 
\begin{equation}\label{PGC:e4}\begin{aligned}
A(w,v) &\ds := \big(DI^{\beta}_\theta w,Dv \big)_{L^2(0,1)} = (Dv,Dv)_{L^2(0,1)} =\|Dv\|^2_{L^2(0,1)}\\
& \ds \geq \min \left \{\f1{2},\f1{2C_0} \right \}^2 \|v\|^2_{H^1(0,1)} \\
& \ds \ge C_3 \min \left \{\f1{2},\f1{2C_0} \right \}^2 \|v\|_{H^1(0,1)}\|w\|_{H^{1-\beta}(0,1)}.
\end{aligned}\end{equation}
We thus prove the first inequality in (\ref{PGC:e1}) with $\kappa = C_3 \min \{1/2,1/(2C_0) \}^2$.

For each $ v \in H^1_0(0,1) \setminus \{0\}$, it follows from Theorem \ref{thm:IntOper} that there exists a unique $w \in H^{1-\beta}_0(0,1)$ such that (\ref{PGC:e3}) holds. Then (\ref{PGC:e4}) shows that the second inequality in (\ref{PGC:e1}) holds. We apply Babu\u{s}ka-Lax-Milgram theorem \cite{Bab,XuZik} to finish the proof.
\end{proof}

The following corollary characterizes the weak solution of the Petrov-Galerkin formulation (\ref{PG:e1}) in terms of the weak solution to a Galerkin formulation of a canonical second-order diffusion equation.

\begin{corollary}\label{cor:Char_PGC} Assume that the conditions of Theorem \ref{thm:PGC} hold. Let $w_l^c$ and $w_r^c$ be defined as in Theorem \ref{thm:Decomp}. Then $u \in H^{1-\beta}_0(0,1)$ is the weak solution of the Petrov-Galerkin formulation (\ref{PG:e1}) if and only if $w \in H^1_0(0,1)$ defined by
\begin{equation}\label{PGC:e5}
I^{\beta}_\theta u - (1-\theta)({}_0^rI^{\beta}_1u)w_l^c - \theta({}_0^lI^{\beta}_1u)w_r^c = w
\end{equation}
is the weak solution of the Galerkin formulation
$$(Dw,Dv)_{L^2(0,1)} = \langle f,v \rangle, \quad  \forall v \in H^1_0(0,1).$$
\end{corollary}

\begin{proof} By Theorem \ref{thm:IntOper}, $u \in H^{1-\beta}_0(0,1)$ if and only if there is a functions $w \in H^1_0(0,1)$ such that (\ref{PGC:e5}) holds. The proof resides in the fact that for $K \equiv 1$, $A(u,v) = (Dw,Dv)_{L^2(0,1)}$ for any $v \in H^1_0(0,1)$.
\end{proof}

\begin{remark}
Despite that they are both formulated for a constant-coefficient analogue of problem (\ref{Model:e1}), the Petrov-Galerkin formulation (\ref{PG:e1}) is different from the Galerkin formulation (\ref{Prv:e1}) in that the latter is defined on the product space $H^{1-\f{\beta}2}_0(0,1) \times H^{1-\frac{\beta}2}_0(0,1)$ for any given $f \in H^{-(1-\frac{\beta}2)}(0,1)$ with $0 < \beta < 1$ while the former is defined on a different product space $H^{1-\beta}_0(0,1) \times H^1_0(0,1)$ for
a given $f \in H^{-1}(0,1)$ with $0 < \beta < 1/2$. Hence, Theorem \ref{thm:Constant} does not apply to the Petrov-Galerkin formulation (\ref{PG:e1}), to which Theorem \ref{thm:PGC} applies.
\end{remark}

\section{Wellposedness, regularity and characterization of weak solutions to variable-coefficient FDEs}

In this section we study the wellposedness of the Petrov-Galerkin formulation (\ref{PG:e1}) and the characterization of the corresponding weak solutions to problem (\ref{Model:e1}).

We begin by letting $w_l$, $w_r$, and $w_f$ be the weak solutions to the following Galerkin formulation for second-order diffusion equations:
\begin{equation}\label{S6:e1}\begin{aligned}
& (K Dw_l,Dv)_{L^2(0,1)} = 0 \ \ \forall \ v \in H^1_0(0,1), \quad w_l(0)=1, ~w_l(1) = 0;\\[0.025in]
& (K Dw_r,Dv)_{L^2(0,1)} = 0 \ \ \forall \ v \in H^1_0(0,1), \quad w_r(0)=0, ~w_r(1) = 1;\\[0.025in]
& (KDw_f,Dv)_{L^2(0,1)} = \langle f,v \rangle, \ \ \forall \ v \in H^1_0(0,1), \quad w_f(0) = w_f(1) = 0
\end{aligned}
\end{equation}
for a given $f \in H^{-1}(0,1)$. It is well known that these problems have unique solutions \cite{Eva}. The solutions $w_l$ and $w_r$ can be solved in closed form as follows
\begin{equation}\label{S6:e2}
w_l= \left (\int^1_0\frac{1}{K(s)}ds \right )^{-1} \int^1_x\frac{1}{K(s)}ds, \qquad w_r= \left (\int^1_0\frac{1}{K(s)}ds \right )^{-1}\int^x_0\frac{1}{K(s)}ds.
\end{equation}
It is clear that $0 \leq w_l, \ w_r \leq 1$ and $w_l+w_r \equiv 1$, and that $w_l$ and $w_r$ are variable extensions of $w_l^c$ and $w_r^c$ introduced in Theorem \ref{thm:Decomp}.

\begin{theorem}\label{thm:Char} Let $0 < \beta < 1/2$ and $0 \le \theta \le 1$. Then $u$ is a weak solution to the Petrov-Galerkin formulation (\ref{PG:e1}) if and only if $u$ satisfies the following integral equation
\begin{equation}\label{S6:e3}
I^{\beta}_\theta u - (1-\theta)\big({}_0^r I^{\beta}_1u \big)w_l - \theta \big({}_0^lI^{\beta}_1u \big )w_r = w_f.
\end{equation}
\end{theorem}

\begin{proof} Suppose that $u$ satisfies (\ref{S6:e3}). Then we use equations (\ref{S6:e1}) to deduce 
$$\begin{aligned}
A(u,v) & = \big (KD(I^{\beta}_\theta u),Dv \big )_{L^2(0,1)} \\
& = (KDw_f,Dv)_{L^2(0,1)} + (1-\theta) \big({}_0^r I^{\beta}_1 u \big) (KDw_l,Dv)_{L^2(0,1)}\\[0.025in]
& \qquad + \theta \big({}_0^l I^{\beta}_1 u \big) (KDw_r,Dv)_{L^2(0,1)} \\[0.025in]
& = (KDw_f,Dv)_{L^2(0,1)} = \langle f,v\rangle, \quad \forall v \in H^1_0(0,1).
\end{aligned}$$
Thus, $u$ is a weak solution to the Petrov-Galerkin formulation (\ref{PG:e1}). Conversely, let $u$ be a weak solution to the Petrov-Galerkin formulation (\ref{PG:e1}). We define
\begin{equation}\label{S6:e4}
w := I^{\beta}_\theta u - (1-\theta)({}_0^rI^{\beta}_1u)w_l - \theta({}_0^lI^{\beta}_1u)w_r. \end{equation}
Theorem \ref{thm:Range} ensures that $w \in H^1(0,1)$. In addition, we have
$$w(0) = I^{\beta}_\theta u|_{x=0} - (1-\theta)({}_0^r I^{\beta}_1u) = (1-\theta)({}_0^r I^{\beta}_1u) - (1-\theta)({}_0^r I^{\beta}_1u) = 0.$$
Similarly, we have $w(1) = 0$. Thus, $w \in H^1_0(0,1)$. Furthermore,
$$\begin{aligned}
(KDw,Dv)_{L^2(0,1)} & = \big(KD(I^{\beta}_\theta u-(1-\theta)({}_0^rI^{\beta}_1u)w_l - \theta({}_0^lI^{\beta}_1u)w_r), Dv \big)_{L^2(0,1)}\\
& = \big(KD(I^{\beta}_\theta u),Dv\big)_{L^2(0,1)} = \langle f,v \rangle, \quad \forall \ v \in H^1_0(0,1).
\end{aligned}$$
In other words, both $w$ and $w_f$ are the solution to the same Galerkin formulation in (\ref{S6:e1}). By the uniqueness of the weak solution to the problem, we conclude that $w = w_f$. Thus, (\ref{S6:e4}) implies that (\ref{S6:e3}) holds.
\end{proof}

We are now in a position to study the existence and uniqueness of the weak solution to Petrov-Galerkin formulation (\ref{PG:e1}). We apply Theorem \ref{thm:IntOper} to $w_l - w_l^c, w_r - w_r^c \in H^1_0(0,1)$ to conclude that there exist unique $u_l$ and $u_r$  in $H^{1-\beta}_0(0,1)$ such that 
\begin{equation}\label{S6:e5}\left\{\begin{aligned}
& I^{\beta}_\theta u_l - (1-\theta)\big({}_0^r I^{\beta}_1 u_l \big)w_l^c  - \theta \big({}_0^l I^{\beta}_1 u_l \big)w_r^c  = w_l - w_l^c,\\
& I^{\beta}_\theta u_r - (1-\theta)\big({}_0^r I^{\beta}_1 u_r \big)w_l^c - \theta \big({}_0^lI^{\beta}_1u_r \big)w_r^c  = w_r - w_r^c.
\end{aligned}\right.\end{equation}
We note that $w_l^c$ and $w_r^c$ satisfy (\ref{Int:e5}) to obtain
\begin{equation}\label{S6:e5a}\begin{aligned}
&  (D I^{\beta}_\theta u_l, Dv)_{L^2(0,1)} = (Dw_l,Dv)_{L^2(0,1)}, \quad &\forall v \in H^1_0(0,1), \\
&  (D I^{\beta}_\theta u_r, Dv)_{L^2(0,1)} = (Dw_r,Dv)_{L^2(0,1)}, \ \ & \forall v \in H^1_0(0,1).
\end{aligned}\end{equation}
Namely, $u_l, u_r \in H^{1-\beta}_0(0,1)$ are two particular solutions to problem (\ref{PG:e1}) with $K \equiv 1$ and $f = -D^2w_l$ and $-D^2w_r$, respectively. The fact that $w_l + w_r \equiv 1$ implies that $u_l + u_r$ satisfies the Galerkin formulation (\ref{Prv:e1}) with $K \equiv 1$ and $f \equiv 0$. The existence and uniqueness of the weak solution to the Galerkin weak formulation (\ref{Prv:e1}) concludes that $u_l + u_r \equiv 0$. In particular, if $K$ is constant, then $w_l \equiv w_l^c$ and $w_r \equiv w_r^c$. So the right-hand sides of the preceding equations vanish, which implies $u_l = u_r \equiv 0$.

Introduce the following variable analogue to the space $\mathcal{U}$ in (\ref{Int:e5a})
$$\begin{aligned}
\mathcal{V} := & \bigl \{g \in H^1_0(0,1): \exists \phi \in H^{1-\beta}_0(0,1), ~ I^{\beta}_\theta \phi - (1-\theta) \big({}_0^r I^{\beta}_1 \phi \big) w_l - \theta \big ({}_0^l I^{\beta}_1 \phi \big)w_r = g \bigr \}.
\end{aligned}$$

\begin{theorem}\label{thm:V} The space $\mathcal{V} = H^1_0(0,1)$, provided the following condition holds
\begin{equation}\label{V:e1}
1 + \theta \big({}_0^l{I}^{\beta}_1 u_l \big) - (1-\theta) \big({}_0^rI^{\beta}_1u_l \big) \neq 0, \quad 0 < x < 1
\end{equation}
or equivalently
\begin{equation}\label{V:e1a} 
1 - \big({}_0^l{I}^{\beta}_1 u_r \big) + (1-\theta) \big ({}_0^rI^{\beta}_1u_r \big)  \neq 0, \quad 0 < x < 1.
\end{equation}
\end{theorem}

\begin{proof} We conclude from equation (\ref{Int6:e1}) in Theorem \ref{thm:IntOper} that for each $g \in H^1_0(0,1)$, there exists a unique $\phi \in H^{1-\beta}_0(0,1)$ such that
\begin{equation}\label{V:e1b}\begin{aligned}
I^{\beta}_\theta \phi &= g + (1-\theta) \big({}_0^r I^{\beta}_1 \phi \big) w_l^c + \theta \big ({}_0^l I^{\beta}_1 \phi)w_r^c\\
&= g + (1-\theta) c_l w_l + \theta c_r w_r - (1-\theta) c_l (w_l-w_l^c) - \theta c_r (w_r-w_r^c)\\
&\quad \quad +(1-\theta)({}_0^r I^{\beta}_1 \phi - c_l)w_l^c + \theta ({}_0^l I^{\beta}_1 \phi - c_r)w_r^c,
\end{aligned}\end{equation}
where $c_l$ and $c_r$ are constants to be determined. We then evaluate $I^{\beta}_\theta((1-\theta) c_l u_l + \theta c_r u_r)$ and replace the $I^{\beta}_\theta u_l$ and $I^{\beta}_\theta u_r$ in the expressions by the rest of the terms in the equations (\ref{S6:e5}) to obtain
\begin{equation}\label{V:e1c}\begin{array}{l}
\hspace{-0.05in} I^{\beta}_\theta((1-\theta) c_l u_l + \theta c_r u_r) \\ 
~= (1-\theta)c_l \bigl [  (1-\theta)({}_0^r I^{\beta}_1 u_l)w_l^c + \theta({}_0^l I^{\beta}_1 u_l)w_r^c + w_l - w_l^c \bigr]\\ [0.025in]
\quad + \theta c_r \bigl [ (1-\theta)({}_0^r I^{\beta}_1 u_r)w_l^c + \theta({}_0^lI^{\beta}_1u_r)w_r^c + w_r - w_r^c \bigr]\\[0.025in]
~=  (1-\theta) c_l (w_l - w_l^c) + \theta c_r (w_r - w_r^c)  \\[0.025in]
\quad + (1-\theta)\big ({}_0^r I^{\beta}_1 ( (1-\theta) c_l u_l + \theta c_r u_r) \big)  w_l^c + \theta \big ({}_0^l I^{\beta}_1 ( (1-\theta) c_l u_l + \theta c_r u_r) \big)  w_r^c.
\end{array}\end{equation}
We sum equations (\ref{V:e1b}) and (\ref{V:e1c}) and cancel the corresponding terms to get
\begin{equation}\label{V:e2}\begin{array}{l}
\hspace{-0.15in} I^{\beta}_\theta ( \phi+(1-\theta) c_l u_l + \theta c_r u_r ) \\[0.05in] 
=  g + (1-\theta) c_l w_l + \theta c_r w_r +  (1-\theta)\big({}_0^r I^{\beta}_1 (\phi +  (1-\theta) c_l u_l + \theta c_r u_r) - c_l \big)  w_l^c \\[0.05in]
~ + \theta \big({}_0^l I^{\beta}_1 ( \phi + (1-\theta) c_l u_l + \theta c_r u_r) - c_r \big)  w_r^c.
\end{array}\end{equation}
Comparing this equation with the definition of the space $\mathcal{V}$ we conclude that to eliminate the last two terms on the right-hand side we shall choose $c_l$ and $c_r$ to satisfy the equations
\begin{equation}\label{V:e3}
\left\{\begin{aligned}
& {}_0^r I^{\beta}_1 (\phi +  (1-\theta) c_l u_l + \theta c_r u_r) - c_l = 0,\\
& {}_0^l I^{\beta}_1 (\phi + (1-\theta) c_l u_l + \theta c_r u_r) - c_r = 0.
\end{aligned}\right. \end{equation}
Equivalently,
$$\ds \left [ \begin{array}{cc}
1 - (1-\theta) ({}_0^r I^{\beta}_1 u_l) & - \theta({}_0^r I^{\beta}_1 u_r) \\[0.075in]
-(1-\theta) ({}_0^l I^{\beta}_1 u_l)     &  1 - \theta ({}_0^l I^{\beta}_1 u_r)
\end{array} \right ]
\left [ \begin{array}{l}
c_l \\ [0.1in]
c_r
\end{array} \right ]
=\left [ \begin{array}{l}
{}_0^r I^{\beta}_1 \phi  \\ [0.075in]
{}_0^l I^{\beta}_1 \phi
\end{array} \right]. $$
This linear system has a unique solution for each given $\phi \in H^{1-\beta}_0(0,1)$ if and only if its coefficient matrix is nonsingular, i.e.,
$$\big(1-(1-\theta)({}_0^rI^{\beta}_1u_l)\big)\big(1-\theta({}_0^l I^{\beta}_1u_r)\big)- (1-\theta)\theta({}_0^r{I}^{\beta}_1u_r)({}_0^lI^{\beta}_1u_l) \neq 0. $$
We incorporate the condition $u_r = - u_l$ into this equation to arrive at (\ref{V:e1}).

Under the condition (\ref{V:e1}), we let $w := \phi + (1-\theta) c_l u_l + \theta c_r u_r$. Then equation (\ref{V:e3}) can be rewritten as
$$c_l = {}_0^r I^{\beta}_1 w, \qquad c_r = {}_0^lI^{\beta}_1 w$$
and (\ref{V:e2}) can be expressed in the form
$$I^{\beta}_\theta w = g + (1-\theta) ({}_0^r I^{\beta}_1 w)w_l + \theta ({}_0^lI^{\beta}_1 w) w_r.$$
In other words, for the given $g \in H^1_0(0,1)$ we have found a function $w \in H^{1-\beta}_0(0,1)$ such that the preceding equation holds. Hence, $g \in \mathcal{V}$. That is, $\mathcal{V}=H^1_0(0,1)$.
\end{proof}

\begin{theorem}\label{thm:Solv} Given $f \in H^{-1}(0,1)$, the Petrov-Galerkin formulation (\ref{PG:e1}) has at least one weak solution, provided that the condition (\ref{V:e1}) holds. Furthermore, the solution is unique if and only if condition (\ref{V:e1}) holds. 
\end{theorem}

\begin{proof} Assume that condition (\ref{V:e1}) holds. Let $w_f \in H^1_0(0,1)$ be the weak solution to the third equation in (\ref{S6:e1}). By Theorem \ref{thm:V}, there exists at least one $u \in H^{1-\beta}_0(0,1)$ such that
$$I^{\beta}_\theta u - (1-\theta) ({}_0^r I^{\beta}_1 u) w_l - \theta({}_0^l I^{\beta}_1 u)w_r = w_f.$$
Theorem \ref{thm:Char} concludes that $u$ is a weak solution to the Petrov-Galerkin formulation (\ref{PG:e1}).

The solution to the Petrov-Galerkin formulation (\ref{PG:e1}) is unique if and only if the corresponding homogeneous formulation has only the trivial solution. By Theorem \ref{thm:Char}, this is equivalent to that the following integral equation
\begin{equation}\label{Uniq:e1}
I^{\beta}_\theta u = (1-\theta)({}_0^r I^{\beta}_1 u)w_l + \theta({}_0^lI^{\beta}_1u)w_r
\end{equation}
has only the trivial solution.

We claim that any solution $u$ to the integral equation (\ref{Uniq:e1}) can be expressed as a linear combination of $u_l$ and $u_r$ introduced in (\ref{S6:e5}). In fact, (\ref{Uniq:e1}) can be rewritten as
\begin{equation}\label{Uniq:e2}\begin{aligned}
I^{\beta}_\theta u & =(1-\theta)({}_0^r I^{\beta}_1 u)(w_l-w_l^c) + \theta({}_0^l I^{\beta}_1 u)(w_r-w_r^c)\\[0.025in]
& ~+(1-\theta)({}_0^r I^{\beta}_1 u)w_l^c + \theta( {}_0^lI^{\beta}_1 u) w_r^c.
\end{aligned}\end{equation}
We use (\ref{S6:e5}) and (\ref{Uniq:e2}) to find that
$$\begin{aligned}
& I^{\beta}_\theta (u-(1-\theta)({}_0^r I^{\beta}_1 u)u_l - \theta({}_0^lI^{\beta}_1 u)u_r) \\
& \quad = I^{\beta}_\theta u-(1-\theta)({}_0^r I^{\beta}_1 u) I^{\beta}_\theta u_l - \theta({}_0^l I^{\beta}_1 u) I^{\beta}_\theta u_r \\
& \quad = (1-\theta)({}_0^r I^{\beta}_1 u)(w_l-w_l^c) + \theta({}_0^l I^{\beta}_1 u)(w_r-w_r^c)\\[0.025in]
& \qquad + (1-\theta)({}_0^r I^{\beta}_1 u)w_l^c + \theta( {}_0^lI^{\beta}_1 u) w_r^c\\
&\qquad - (1-\theta)({}_0^r I^{\beta}_1 u) [(1-\theta)({}_0^r I^{\beta}_1 u_l)w_l^c + \theta({}_0^l I^{\beta}_1 u_l)w_r^c + w_l - w_l^c ] \\
&\qquad - \theta({}_0^l I^{\beta}_1 u) [(1-\theta) ({}_0^r I^{\beta}_1 u_r)w_l^c + \theta({}_0^lI^{\beta}_1u_r)w_r^c + w_r - w_r^c]\\
&\quad = (1-\theta)({}_0^r I^{\beta}_1 u)w_l^c + \theta( {}_0^lI^{\beta}_1 u) w_r^c\\
&\qquad - (1-\theta)({}_0^r I^{\beta}_1 u) [(1-\theta)({}_0^r I^{\beta}_1 u_l)w_l^c + \theta({}_0^l I^{\beta}_1 u_l)w_r^c ] \\
&\qquad - \theta({}_0^l I^{\beta}_1 u) [(1-\theta) ({}_0^r I^{\beta}_1 u_r)w_l^c + \theta({}_0^lI^{\beta}_1u_r)w_r^c ]\\
&\quad =  (1-\theta) \big[\ {}_0^r I^{\beta}_1 \big(u-(1-\theta)({}_0^r I^{\beta}_1 u)u_l - \theta({}_0^l I^{\beta}_1 u)u_r \big) \big] w_l^c \\
&\qquad +  \theta \big[{}_0^l I^{\beta}_1 \big( u-(1-\theta)({}_0^r I^{\beta}_1 u)u_l - \theta({}_0^l I^{\beta}_1 u)u_r \big) \big] w_r^c.
\end{aligned}$$
In other words, $\phi := u-(1-\theta)({}_0^r I^{\beta}_1 u)u_l - \theta({}_0^l I^{\beta}_1 u)u_r$
satisfies the homogeneous integral equation
\begin{equation}\label{Uniq:e3}
I^{\beta}_\theta \phi  - (1-\theta)({}_0^rI^{\beta}_1\phi)w_l^c - \theta({}_0^lI^{\beta}_1 \phi)w_r^c = 0.
\end{equation}
By Theorem \ref{thm:IntOper}, i.e., (\ref{Int6:e1a}) with $g=0$, this equation has only the trivial solution $\phi \equiv 0$. That is,
$$u = (1-\theta) ({}_0^r I^{\beta}_1 u) u_l + \theta({}_0^l I^{\beta}_1 u)u_r.$$
Thus, we have proved the claim.

We are now in a position to prove that (\ref{Uniq:e1}) has only the trivial solution. Let $u$ be any linear combination of $u_l$ and $u_r$. Recall that $u_l + u_r \equiv 0$, we have $u = c_l u_l$ with $c_l$ being an undetermined constant. We prove that $c_l$ must be zero under the condition (\ref{V:e1}). As a matter of fact, equations (\ref{S6:e5}) and (\ref{Uniq:e2}) respectively reduce to
$$I^{\beta}_\theta u = c_l  [ (1-\theta)({}_0^r I^{\beta}_1 u_l)w_l^c  + \theta({}_0^l I^{\beta}_1 u_l)w_r^c + w_l - w_l^c]$$
and
$$\begin{aligned}
I^{\beta}_\theta u & =c_l\big[(1-\theta)({}_0^r I^{\beta}_1 u_l)(w_l-w_l^c)
+ \theta({}_0^l I^{\beta}_1 u_l)(w_r-w_r^c)\\[0.025in]
& \quad +(1-\theta)({}_0^r I^{\beta}_1 u_l)w_l^c + \theta( {}_0^lI^{\beta}_1 u_l) w_r^c\big].
\end{aligned}$$
Subtracting the first equation from the second yields
$$ c_l\big[(1-\theta)({}_0^r I^{\beta}_1 u_l)(w_l-w_l^c)
+ \theta({}_0^l I^{\beta}_1 u_l)(w_r-w_r^c)\big] = c_l(w_l - w_l^c).$$
We utilize the fact that $w_r+w_l=w^c_r+w^c_l=1$ to rewrite this equation as
$$\bigl [1 + \theta({}_0^l I^{\beta}_1 u_l) - (1-\theta)({}_0^r I^{\beta}_1 u_l) \bigr] (w_l-w_l^c)c_l
= 0, \qquad 0 \leq x \leq 1.$$
This equation has only the trivial solution $(w_l-w_l^c)c_l = 0$ if and only if condition (\ref{V:e1}) holds. As we are considering the case of a variable diffusivity coefficient $K$, $w_l \not\equiv w_l^c$. Otherwise, it is clear from (\ref{S6:e2}) and the fact that $w_l^c = x$ that $K$ would be a constant. Hence,
$c_l = 0$ if and only if condition (\ref{V:e1}) holds.
\end{proof}

\begin{theorem}\label{thm:PG} Consider problem (\ref{Model:e1}) with $0 < \beta < 1/2$, $0 \le \theta \le 1$, and that $K \in L^\infty(0,1)$. In addition,  assume that condition (\ref{V:e1}) holds. Then the bilinear form $A(\cdot,\cdot)$ in (\ref{PG:e1}) is bounded and weakly coercive on the space $H^{1-\beta}_0(0,1) \times H^1_0(0,1)$. That is, the inequalities in (\ref{PGC:e1}) hold. Then, for any $f \in H^{-1}(0,1)$, the Petrov-Galerkin formulation (\ref{PG:e1}) has a unique solution $u \in H^{1-\beta}_0(0,1)$ such that the stability estimate (\ref{PGC:e2}) holds.
\end{theorem}

\begin{proof} It is clear that the bilinear form $A(\cdot,\cdot)$ is bounded on the space $H^{1-\beta}_0(0,1) \times H^1_0(0,1)$. To prove its weak coercivity, 
for any given $v \in H^1_0(0,1)\setminus\{0\}$, by Theorem \ref{thm:Solv} there exists a unique $w \in H^{1-\beta}_0 (0,1)$ such that
\begin{equation}\label{PG:e2}
v = I^{\beta}_\theta w - (1-\theta)({}_0^rI^{\beta}_1 w )w_l - \theta ({}_0^lI^{\beta}_1 w) w_r.
\end{equation}
A direct evaluation reveals
\begin{equation}\label{PG:e3}\begin{aligned}
A(w,v) &= (KD(I^{\beta}_\theta w),Dv)_{L^2(0,1)} \\
& = (KDv,Dv)_{L^2(0,1)} + (1-\theta)({}_0^rI^{\beta}_1w ) (KDw_l,Dv)_{L^2(0,1)} \\
& \quad + \theta ({}_0^lI^{\beta}_1w ) (KDw_r,Dv)_{L^2(0,1)} \\
& = (KDv,Dv)_{L^2(0,1)} \ge K_m \| Dv\|_{L^2(0,1)}^2\\
& \ge \kappa \|v\|_{H^1(0,1)} \|w\|_{H^{1-\beta}_0(0,1)}
\end{aligned}\end{equation}
similarly to the derivation of (\ref{PGC:e4}). Thus, the first estimate in (\ref{PGC:e1}) holds.

To prove the second estimate in (\ref{PGC:e1}), let $P_K$ be the projection operator from $H^1(0,1)$ onto $H^1_0(0,1)$ defined as follows: for any $ w \in H^1(0,1)$, find $P_K w \in H^1_0(0,1)$ such that
$$(KD(P_K w),Dv)_{L^2(0,1)} = (KDw,Dv)_{L^2(0,1)}, \quad \forall\ v \in H^1_0(0,1).$$
Clearly we have
$$\|D(P_K w) \|_{L^2(0,1)} \leq \sqrt{\frac{K_M}{K_m}}\|Dw\|_{L^2(0,1)}, \quad \forall\ w \in H^1(0,1).$$

Similarly to the proof in Theorem \ref{thm:IntOper}, we can prove that $P_K I^{\beta}_\theta$ is a bounded linear bijection from $H^{1-\beta}_0(0,1)$ onto $H^1_0(0,1)$ and has a bounded inverse, provided that condition (\ref{V:e1}) holds. Thus, for any $v \in H^1_0(0,1)$, there exists a unique $w \in H^{1-\beta}_0(0,1)$ which satisfies (\ref{PG:e2}) and
$$ \| w \|_{H^{1-\beta}(0,1)} \leq C \| v \|_{H^1(0,1)}.$$
We get from (\ref{PG:e3}) that
$$A(w,v) = (KDv,Dv)_{L^2(0,1)} \ge K_m \| Dv\|^2_{L^2(0,1)} \ge \kappa \| w \|_{H^{1-\beta}(0,1)} \| v \|_{H^1(0,1)}.$$
We thus prove (\ref{PG:e3}). The rest of the theorem is a direct application of Babu\u{s}ka-Lax-Milgram theorem \cite{Bab,XuZik}.
\end{proof}

\begin{theorem}\label{thm:Regularity} Assume that the conditions in Theorem \ref{thm:PG} hold. Furthermore, we assume that 
$K \in C^{m+1}[0,1]$ and $f \in H^m(0,1)$ with $m$ being a nonnegative integer, then the following regularity estimates hold for the weak solution $u \in H^{1-\beta}_0(0,1)$ to the Petrov-Galerkin formulation (\ref{PG:e1})
\begin{equation}\label{Reg:e1}
\big \| I^\beta_\theta u \big \|_{H^{m+2}(0,1)} \leq C \big (\| K\|_{C^{m+1}[0,1]} + \|f\|_{H^m(0,1)} \big).
\end{equation}
In other words, the weighted high-order Riemann-Liouville derivatives $D^{m+2} I^\beta_\theta u$ exist and can be bounded by the high-order norms of the data of the variable-coefficient FDE (\ref{Model:e1}).
\end{theorem}

\begin{proof} Let $u$ be the weak solution to the Petrov-Galerkin formulation (\ref{PG:e1}). By Theorem \ref{thm:Char}, $u$ satisfies (\ref{S6:e3}). Under the conditions of the theorem, it follows from the classical theory of second-order diffusion equations \cite{Eva} that $w_f \in H^{m+2}(0,1) \cap H^1_0(0,1)$ satisfies
$$\| w_f\|_{H^{m+2}(0,1)} \leq C \| f \|_{H^m(0,1)}.$$
A direct differentiation of (\ref{S6:e2}) concludes that there is a positive constant $C$ such that 
$$\|w_l\|_{H^{m+2}(0,1)} + \|w_r\|_{H^{m+2}(0,1)} \le C \|K\|_{C^{m+1}[0,1]}.$$

We differentiate (\ref{S6:e3}) successively for $k$ times with $0 \le k \le m+2$ to obtain 
\begin{equation}\label{Reg:e2}\begin{array}{rl}
D^k I^{\beta}_\theta u &= (1-\theta)({}_0^r I^{\beta}_1u)D^k w_l
+\theta({}_0^lI^{\beta}_1u)D^k w_r + D^k w_f.
\end{array}\end{equation}
We apply Young's inequality and Sobolev's embedding theorem \cite{AdaFou,ErvRoo05} and the estimate (\ref{PG:e2}) to deduce that
$$
\bigl | {}_0^r I^{\beta}_1u \big | + \big | {}_0^lI^{\beta}_1u \big | \le C(\beta) \| u\|_{L^\infty(0,1)} \le C \|u\|_{H^{1-\beta}(0,1)} \le C\| f\|_{H^{-1}(0,1)}. 
$$
We thus come up with the following bounds for $2 \le k \le m+2$
$$\begin{array}{l}
\|D^k I^{\beta}_\theta u\|_{L^2(0,1)} \\[0.05in]
\qquad \leq C \bigl ( (1-\theta) \big |{}_0^r I^{\beta}_1 u \big | \|D^k w_l\|_{C^{k-1}[0,1]}
+ \theta \big |{}_0^lI^{\beta}_1u \big | \| w_r \|_{C^{k-1}[0,1]} + \|w_f\|_{H^k(0,1)} \big )\\[0.05in]
\qquad \le  C \bigl (\| K\|_{C^{m-1}[0,1]} + \|f\|_{H^{k-2}(0,1)} \big).
\end{array}$$
We thus finish the proof of the theorem.
\end{proof}

\section{Application of the theory}

In \S 6 we proved the existence, uniqueness, regularity and characterization of the weak solution to the homogeneous Dirichlet boundary-value problem of two-sided variable-coefficient conservative FDEs (\ref{Model:e1}), the weak coercivity and stability of the corresponding Petrov-Galerkin formulation (\ref{PG:e1}), under the condition that (\ref{V:e1}) holds.

While (\ref{V:e1}) looks somewhat nonconventional, in this section we look at some important special cases of (\ref{Model:e1}) for which (\ref{V:e1}) holds.

\begin{theorem}\label{thm:Oneside} The conclusions of Theorems \ref{thm:Solv}--\ref{thm:Regularity} hold for the one-sided analogue of problem (\ref{Model:e1}).
\end{theorem}

\begin{proof} We need only to prove that (\ref{V:e1}) holds. By symmetry, we need only consider the case of $\theta=1$ when (\ref{V:e1}) reduces to
$$ 1+{}_0^lI^{\beta}_1u_l \neq 0, \quad 0 < x < 1. $$
On the other hand, the first equation in (\ref{S6:e5}) reduces to
$$\begin{aligned}
{}_0^l I^{\beta}_x u_l &= w_l - w_l^c +({}_0^lI^{\beta}_1u_l)w_r^c = w_l - 1 + \big (1 + {}_0^lI^{\beta}_1u_l \big)x \\
&= - w_r + \big (1+{}_0^lI^{\beta}_1u_l \big)x.
\end{aligned}$$
We apply  $D ~{}_0^lI^{1-\beta}_x$ on both sides of the equation and use Lemmas \ref{lem:2C4} and \ref{lem:2C7} as well as (\ref{S6:e2}) to obtain
\begin{equation}\label{Appl:e1}\begin{array}{rl}
u_l & = - D ~{}_0^lI^{1-\beta}_x w_r + \bigl ( 1+({}_0^lI^{\beta}_1u_l) \bigr ) \bigl (D ~{}_0^lI^{1-\beta}_x x \bigr) \\[0.05in]
& = - {}_0^lI^{1-\beta}_x Dw_r + \bigl (1+({}_0^lI^{\beta}_1u_l) \bigr) \bigl ({}_0^lI^{1-\beta}_x 1 \bigr)\\[0.05in]
& \ds = -\f1{\Gamma(1-\beta)} \int_0^x \left [ \f1{(x-s)^\beta}\left (\int_0^1 \f{1}{K(\theta)} d\theta \right )^{-1} \f1{K(s)} \right ] ds \\[0.15in]
& \ds \qquad + \bigl (1+({}_0^lI^{\beta}_1u_l) \bigr) \f1{\Gamma(1-\beta)} \int_0^x \f1{(x-s)^\beta}ds.
\end{array}\end{equation}
We let $s = xt$ to get  
\begin{equation}\label{Appl:e2}
\begin{aligned}\ds \f1{\Gamma(1-\beta)} \int_0^x \f1{(x-s)^\beta}ds &= \f{x^{1-\beta}}{\Gamma(1-\beta)} \int_0^1 \f1{(1-t)^\beta}dt \\
& = \f{x^{1-\beta}}{\Gamma(1-\beta)} B(1,1-\beta) =  \f{x^{1-\beta}}{\Gamma(2-\beta)}\end{aligned}\end{equation}
with $B(\cdot,\cdot)$ being the Beta function. We note from (\ref{S6:e5}) that $u_l \in H^{1-\beta}_0(0,1)$ and let $x=1$ in (\ref{Appl:e1}) and (\ref{Appl:e2}) to obtain
$$\ds 1+{}_0^lI^{\beta}_1u_l = \Gamma(1-\beta) \int_0^x \left [ \f1{(x-s)^\beta}\left (\int_0^1 \f{1}{K(\theta)} d\theta \right )^{-1}
\f1{K(s)} \right] ds  > 0.$$
Thus, (\ref{V:e1}) holds for problem (\ref{Model:e1}) with $\theta = 1$. We can similarly prove that (\ref{V:e1}) holds for problem (\ref{Model:e1}) with $\theta = 0$.
\end{proof}

The following theorem shows that problem (\ref{Model:e1}) is well posed if the diffusivity coefficient $K$ is a perturbation of a constant.
\begin{theorem} \label{thm:Perturb} If the diffusivity coefficient $K$ in (\ref{Model:e1}) is a perturbation from a constant in the $L^2$ sense, i.e., there exists a constant $\varepsilon_0  > 0$ for which 
\begin{equation}\label{Perturb:e1}
\Big \| \Big(\int^1_0\frac{1}{K(s)}ds \Big)^{-1} \frac{1}{K} -1 \Big \|_{L^2(0,1)} < \varepsilon_0,
\end{equation}
then the conclusions of Theorems \ref{thm:Solv}--\ref{thm:Regularity} hold.
\end{theorem}

\begin{proof} We notice that $Dw_r = \big(\int^1_0\frac{1}{K(s)}ds \big)^{-1} \frac{1}{K} $ satisfies that $\int_0^1 Dw_r dx = 1$. Hence, the minimizer $P^*_0$ of the approximation 
$$\ds P_0^* = \mathrm{argmin}_{P_0 \in \mathbb{R}}  \big \| Dw_r - P_0 \big \|_{L^2(0,1)} $$
must be given by its first Fourier coefficient or $L^2$ projection to the space of constant $P^*_0 = \int_0^1 Dw_r 1 dx = 1$. Furthermore, $(1, Dv)_{L^2(0,1)}=0$ for any $ v \in H^1_0(0,1)$. Hence, the second equation in (\ref{S6:e5a}) still holds when $Dw_r$ in its right-hand side is replaced by $Dw_r - 1$ for any constant $C$. By the stability estimate (\ref{PGC:e2}) for Theorem \ref{thm:PG} there exists a constant $C_4 > 0$ such that
$$\ds \big \| u_r \big \|_{H^{1-\beta}(0,1)} \leq C_4 \Big \| \Big(\int^1_0\frac{1}{K} dx \Big)^{-1} \frac{1}{K} -1 \Big \|_{L^2(0,1)}. $$
Furthermore, we notice that 
$$\begin{aligned}
&\bigl | \theta \big({}_0^l{I}^{\beta}_1 u_l \big) - (1-\theta) \big({}_0^rI^{\beta}_1u_l \big) \big | \\
&\quad \ds = \f1{\Gamma(\beta)} \Bigl | \theta \int_0^1 \frac{u_l(s)}{(1-s)^{1-\beta}} ds - (1-\theta) \int_0^1 \frac{u_l(s)}{s^{1-\beta}} ds \Big | \\
&\quad \ds \le \f1{\Gamma(\beta+1)} \| u_l \|_{L^\infty(0,1)} \le C_5 \| u_l \|_{H^{1-\beta}_0(0,1)}\end{aligned}$$
where we have used Sobolev's embedding theorem at the last step.

Consequently, we obtain
$$\begin{aligned}
&1 + \theta \big({}_0^l{I}^{\beta}_1 u_l \big) - (1-\theta) \big({}_0^rI^{\beta}_1u_l \big) \\
& \quad \ds \ge 1 - \big | \theta \big({}_0^l{I}^{\beta}_1 u_l \big) - (1-\theta) \big({}_0^rI^{\beta}_1u_l \big) \big | \\
& \quad \ds \ge 1 - C_5 \| u_l \|_{H^{1-\beta}_0(0,1)} = 1 - C_5 \| u_r \|_{H^{1-\beta}_0(0,1)}\\
& \quad \ds \ge 1 - C_4 C_5 \Big \| \Big(\int^1_0\frac{1}{K} dx \Big)^{-1} \frac{1}{K} -1 \Big \|_{L^2(0,1)}.\end{aligned}$$
We finish the proof of the theorem by selecting $\varepsilon_0=1/(C_4C_5)$. 
\end{proof}


\end{document}